\documentclass[journal,10pt,onecolumn]{IEEEtran}
\usepackage{amsmath}
\usepackage{amsfonts}
\usepackage{amsthm}
\usepackage{dsfont}
\usepackage{epsfig}
\usepackage{upgreek}
\usepackage{graphicx}

\usepackage{setspace}

\theoremstyle{definition}
\newtheorem{theorem}{Theorem}[section]
\newtheorem{lemma}[theorem]{Lemma}

\newtheorem{definition}[theorem]{Definition}

\newcommand{\Diag}{\mbox{Diag}}
\newcommand{\Abs}{\mbox{Abs}}
\newcommand{\Log}{\mbox{{LogAbs}}}

\newcommand{\Exp}{\mbox{{Exp}}}

\newcommand{\adist}[2]{\mbox{{adist}\,[\,#1,~#2\,]}}

\newcommand{\Jn}{\mbox{${\bf J}_n$}}
\newcommand{\Imn}{\mbox{${\bf I}_n$}}

\newcommand{\Jm}{\mbox{${\bf J}_m$}}
\newcommand{\Imm}{\mbox{${\bf I}_m$}}
\newcommand{\Imat}{\mbox{${\bf I}$}}

\newcommand{\Amn}{\mbox{${\bf A}$}}
\newcommand{\Amni}{\mbox{${\bf A}^{\!\mbox{\tiny -1}}$}}

\newcommand{\Lmn}{\mbox{${\bf L}$}}

\newcommand{\Dmp}{\mbox{${\bf D}_{\mbox{\tiny +}}$}}
\newcommand{\Dmu}{\mbox{${\bf D}_{\mbox{\tiny u}}$}}
\newcommand{\Dmv}{\mbox{${\bf D}_{\mbox{\tiny v}}$}}
\newcommand{\Dmpi}{\mbox{${\bf D}_{\tiny +}^{\tiny -1}$}}

\newcommand{\Dmui}{\mbox{${\bf D}_{\tiny u}^*$}}

\newcommand{\Emp}{\mbox{${\bf E}_{\mbox{\tiny +}}$}}
\newcommand{\Empi}{\mbox{${\bf E}_{\tiny +}^{\tiny -1}$}}
\newcommand{\Emu}{\mbox{${\bf E}_{\mbox{\tiny u}}$}}

\newcommand{\Emui}{\mbox{${\bf E}_{\tiny u}^*$}}

\newcommand{\omc}{\mbox{${{\mathds 1}}_m$}}
\newcommand{\omr}{\mbox{${{\mathds 1}}_m^{\textnormal{\tiny{T}}}$}}
\newcommand{\onc}{\mbox{${{\mathds 1}}_n$}}
\newcommand{\onr}{\mbox{${\mathds 1}_n^{\textnormal{\tiny{T}}}$}}

\newcommand{\znc}{\mbox{${{\bf 0}}_n$}}

\newcommand{\dmc}{\mbox{${\bf d}_m$}}

\newcommand{\emc}{\mbox{${\bf e}_m$}}
\newcommand{\emr}{\mbox{${\bf e}_m^{\textnormal{\tiny{T}}}$}}
\newcommand{\enc}{\mbox{${\bf e}_n$}}
\newcommand{\enr}{\mbox{${\bf e}_n^{\mbox{\tiny{T}}}$}}

\newcommand{\umc}{\mbox{${\bf u}_m$}}

\newcommand{\vmc}{\mbox{${\bf v}_m$}}

\newcommand{\vnc}{\mbox{${\bf v}_n$}}
\newcommand{\vnr}{\mbox{${\bf v}_n^{\mbox{\tiny{T}}}$}}
\newcommand{\stimes}{\mbox{$s_{\times}$}}

\newcommand{\Sm}{\mbox{${\bf S}$}}

\newcommand{\uv}{\mbox{${\bf u}$}}
\newcommand{\vv}{\mbox{${\bf v}$}}
\newcommand{\pv}{\mbox{${\bf p}$}}
\newcommand{\qv}{\mbox{${\bf q}$}}
\newcommand{\xb}{\mbox{${\bf x}$}}
\newcommand{\yb}{\mbox{${\bf y}$}}
\newcommand{\bv}{\mbox{${\bf b}$}}

\newcommand{\Perm}{\mbox{${\bf P}$}}

\newcommand{\Qm}{\mbox{${\bf Q}$}}

\newcommand{\Rm}{\mbox{${\bf R}$}}

\newcommand{\Vm}{\mbox{${\bf V}$}}
\newcommand{\Vmi}{\mbox{${\bf V}^*$}}
\newcommand{\Xm}{\mbox{${\bf X}$}}
\newcommand{\Xmi}{\mbox{${\bf X}^{\mbox{\tiny -1}}$}}
\newcommand{\Xpinv}{\mbox{${\bf X}^{\!\mbox{\tiny -P}}$}}
\newcommand{\Xrinv}{\mbox{${\bf X}^{\!\mbox{\tiny -R}}$}}

\newcommand{\Ym}{\mbox{${\bf Y}$}}
\newcommand{\Ymi}{\mbox{${\bf Y}^{\mbox{\tiny -1}}$}}

\newcommand{\Dm}{\mbox{${\bf D}$}}
\newcommand{\Dmi}{\mbox{${\bf D}^{\mbox{\tiny -1}}$}}
\newcommand{\Dma}{\mbox{${\bf D}_1$}}
\newcommand{\Dmia}{\mbox{${\bf D}_1^{\mbox{\tiny -1}}$}}
\newcommand{\Dmb}{\mbox{${\bf D}_2$}}
\newcommand{\Dmib}{\mbox{${\bf D}_2^{\mbox{\tiny -1}}$}}
\newcommand{\Em}{\mbox{${\bf E}$}}
\newcommand{\Emi}{\mbox{${\bf E}^{\mbox{\tiny -1}}$}}
\newcommand{\Ema}{\mbox{${\bf E}_1$}}
\newcommand{\Emia}{\mbox{${\bf E}_1^{\mbox{\tiny -1}}$}}
\newcommand{\Emb}{\mbox{${\bf E}_2$}}
\newcommand{\Emib}{\mbox{${\bf E}_2^{\mbox{\tiny -1}}$}}

\newcommand{\Fm}{\mbox{${\bf F}$}}
\newcommand{\Fmt}{\mbox{${\bf F}^*$}}
\newcommand{\Gm}{\mbox{${\bf G}$}}
\newcommand{\Gmt}{\mbox{${\bf G}^*$}}
\newcommand{\FP}{\mbox{{\bf F}$_{\mbox{\tiny P}}$}}

\newcommand{\FR}{\mbox{{\bf F}$_{\mbox{\tiny R}}$}}

\newcommand{\dinv}[1]{{#1}^{\mbox{\tiny -D}}}
\newcommand{\linv}[1]{{#1}^{\mbox{\tiny -L}}}
\newcommand{\rinv}[1]{{#1}^{\mbox{\tiny -R}}}
\newcommand{\ginv}[1]{{#1}^{\mbox{\tiny -U}}}
\newcommand{\pinv}[1]{{#1}^{\mbox{\tiny -P}}}
\newcommand{\inv}[1]{{#1}^{\mbox{\tiny -1}}}
\newcommand{\DL}{\mbox{${\cal D}_{\mbox{\tiny L}}$}}
\newcommand{\DG}{\mbox{${\cal S}_{\mbox{\tiny U}}$}}
\newcommand{\DGL}{\mbox{${\cal D}_{\mbox{\tiny U}}^{\mbox{\tiny L}}$}}
\newcommand{\DGR}{\mbox{${\cal D}_{\mbox{\tiny U}}^{\mbox{\tiny R}}$}}
\newcommand{\TR}{\mbox{${\bf\cal T}$}}

\newcommand{\norm}[1]{\mbox{$\lVert #1 \rVert$}}

\newcommand{\Um}{\mbox{${\bf U}$}}
\newcommand{\Umi}{\mbox{${\bf U}^*$}}

\newcommand{\xmc}{\mbox{${\bf x}_m$}}

\newcommand{\ync}{\mbox{${\bf y}_n$}}
\newcommand{\ynr}{\mbox{${\bf y}_n^{\mbox{\tiny{T}}}$}}

\newcommand{\Atinv}{\mbox{${\bf A}^{\overset{\sim}{\!\mbox{\tiny -1}}}$}}
\newcommand{\Adinv}{\mbox{${\bf A}^{\!\mbox{\tiny -D}}$}}
\newcommand{\Apinv}{\mbox{${\bf A}^{\!\mbox{\tiny -P}}$}}
\newcommand{\Aginv}{\mbox{${\bf A}^{\!\mbox{\tiny -U}}$}}

\newcommand{\tinv}[1]{\mbox{${#1}^{\overset{\sim}{\mbox{\tiny -1}}}$}}

\newcommand{\yv}{\mbox{$\hat{\bf y}$}}
\newcommand{\av}{\mbox{$\hat{\bf \uptheta}$}}
\newcommand{\yyv}{\mbox{$\hat{\bf y}'$}}
\newcommand{\aav}{\mbox{$\hat{\bf \uptheta}'$}}

\newcommand{\Lnsv}{\mbox{L$_{\mbox{\scriptsize NSV}}$}}
\newcommand{\LXnsv}{\mbox{$\widetilde{\mbox{L}}_{\mbox{\scriptsize NSV}}$}}
\newcommand{\Jnsv}{\mbox{J$_{\mbox{\scriptsize NSV}}$}}
\newcommand{\JXnsv}{\mbox{$\widetilde{\mbox{J}}_{\mbox{\scriptsize NSV}}$}}
\newcommand{\Rnsv}{\mbox{R$_{\mbox{\scriptsize NSV}}$}}
\newcommand{\RXnsv}{\mbox{$\widetilde{\mbox{R}}_{\mbox{\scriptsize NSV}}$}}

\newcommand{\Lunsv}{\mbox{L$_{\mbox{\scriptsize UNSV}}$}}
\newcommand{\LXunsv}{\mbox{$\widetilde{\mbox{L}}_{\mbox{\scriptsize UNSV}}$}}
\newcommand{\Junsv}{\mbox{J$_{\mbox{\scriptsize UNSV}}$}}
\newcommand{\JXunsv}{\mbox{$\widetilde{\mbox{J}}_{\mbox{\scriptsize UNSV}}$}}
\newcommand{\Runsv}{\mbox{R$_{\mbox{\scriptsize UNSV}}$}}
\newcommand{\RXunsv}{\mbox{$\widetilde{\mbox{R}}_{\mbox{\scriptsize UNSV}}$}}

\begin{document}

\title{{Unit Consistency, Generalized Inverses, and\\
           Effective System Design Methods} }       
\author{
%{\large\bf Jeffrey Uhlmann}\\
%University of Missouri-Columbia}
\IEEEauthorblockN{{\large\bf Jeffrey Uhlmann}}\\
\IEEEauthorblockA{\small University of Missouri-Columbia\\
201 EBW, Columbia, MO 65211\\
Email: uhlmannj@missouri.edu}}
\date{}          
\maketitle
\thispagestyle{empty}

%For singlespaced version
\vspace{-11pt}
%For doublespaced version
%\vspace{-0.5in}

\begin{abstract}
A new generalized matrix inverse is derived which is consistent
with respect to arbitrary nonsingular diagonal transformations,
e.g., it preserves units associated with variables under state
space transformations. Applications of this unit-consistent (UC) 
generalized inverse are examined, including maintenance of
unit consistency as a design principle for promoting and assessing
the functional integrity of complex engineering systems. Results are 
generalized to obtain UC and unit-invariant matrix 
decompositions and illustrative
examples of their use are provided.\\
~\\
\begin{footnotesize}
\end{footnotesize}

\end{abstract}

\section{Introduction}

Many of the benefits of modular system design derive from an assumption 
that each module has been separately tested and verified  so that the 
composite system can be analyzed, tuned, and evaluated at a more 
manageable level of abstraction. However, seemingly benign decisions 
made as part of the design and implementation of a given module can have
significant unanticipated effects on the behavior of a system in which it is 
used. Consider a notional representation of a module that processes 
a parameter vector and a dataset to produce an output:
\begin{center}
     \includegraphics[scale=0.25]{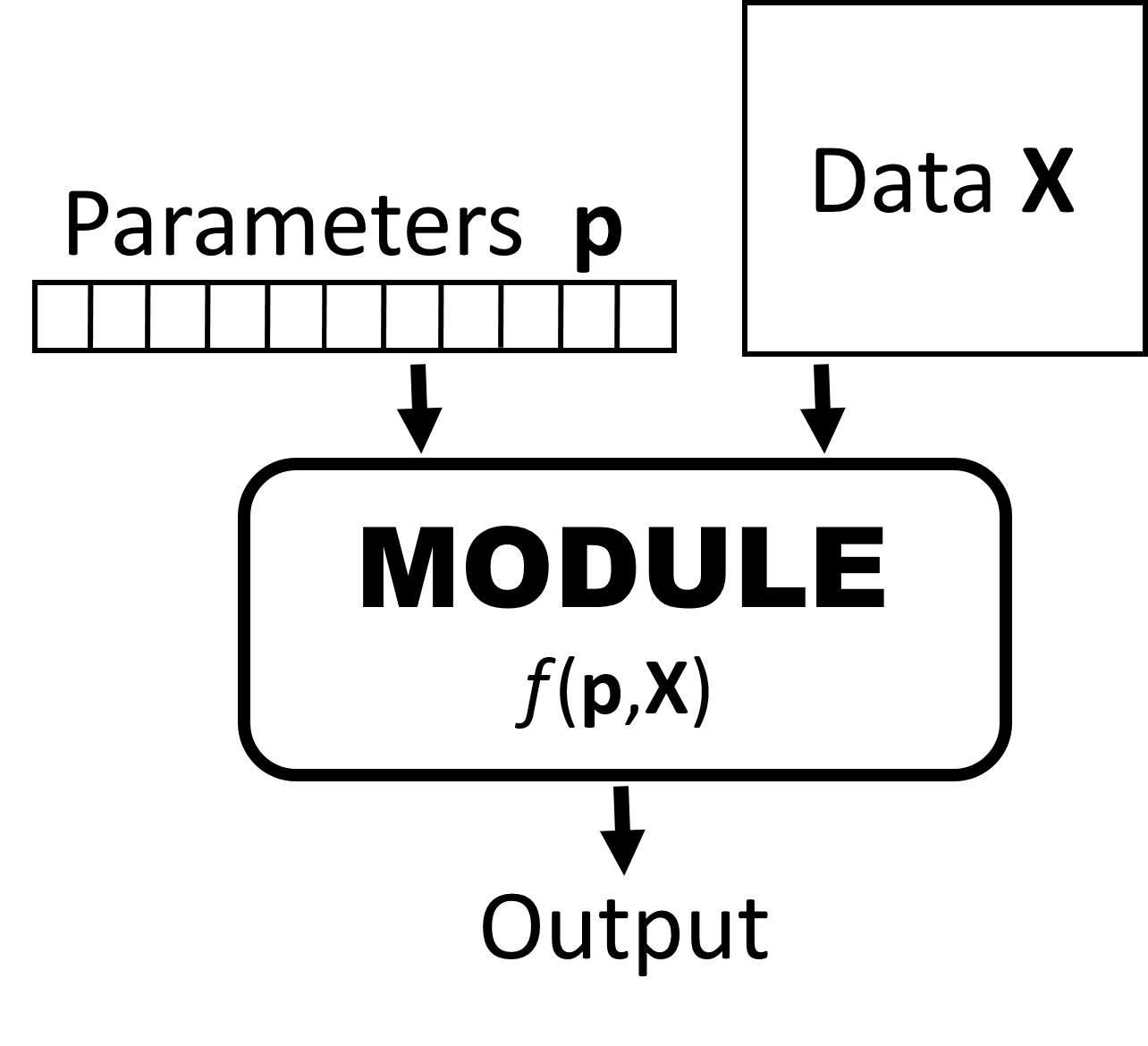}
\end{center}
The function performed by the module could be as simple as 
$f(\pv,\Xm)\doteq\Xm\pv\rightarrow\qv$, where $\Xm$ is interpreted 
to be a matrix and ${\bf p}$ is a column vector that is to be linearly
transformed by $\Xm$ to produce an output vector $\qv$.  
Such a ``linear transformation module'' can be easily implemented
and verified for correctness. For example, if the inputs $\pv$, $\Xm$, 
and $\qv$ are all defined in the same coordinate frame then the 
module can be tested to assess whether it is consistent with
respect to 
an orthogonal/unitary rotation of the coordinate frame,
i.e., that a unitary transformation applied to the input as
$\pv'=\Um\pv$ and $\Xm'=\Um\Xm\Umi$ gives
$f(\pv',\Xm')\rightarrow\Um\qv$. If consistency with
respect to a rotation of the coordinate frame is assumed
to hold at the system level then the integrity of the system 
as a whole can be tested with a change of coordinates
using a random rotation matrix $\Rm$.
More specifically, if an arbitrary rotation of the coordinate 
frame is applied to the inputs then the system should  
produce the same output but in the new rotated 
coordinates. If that does not occur then a fault has 
been detected. 

Note that this kind of consistency testing does not
just detect errors that are specifically related to the
choice of coordinate frame; rather, it detects {\em any}
error that leads to a violation of the consistency 
assumption. In other words, a consistency test 
is analogous to the use of parity bits and
checksums for detecting whether a bit string has been
corrupted: an error is unlikely to preserve
the assumed properties and thus will be
detected with high probability. 

Most complex real-world systems perform
transformations from inputs to outputs that are defined
in an application-specific state space for
which there is only an expectation of consistency with respect
to the choice of units associated with state variables
(e.g., kilometers-per-hour versus meters-per-second) rather
than a unitary mixing of those state variables. The 
appropriate consistency test should therefore assess
whether the system is {\em unit consistent} (UC). 
In the case of the linear transformation module, 
for example, the test for unit consistency might
verify that a diagonal change-of-unit transformation 
$\Dm$ applied as $\pv'=\Dm\pv$ and $\Xm'=\Dm\Xm\Dmi$ gives
$f(\pv',\Xm')\rightarrow\Dm\qv$. In other words, a change
of units for the input to the module should  
produce the same output but in the new units. 

Unit-consistency testing represents a potentially
valuable means for assessing a fundamental aspect of
system integrity that does not rely in any way on a 
qualitative interpretation of a battery of empirical tests 
or comparison against a limited set of ground-truth
benchmarks. Unfortunately, unit-consistency 
testing is not presently viable for the majority of 
complex systems because many modules that 
should be expected to exhibit unit consistency actually 
do not. This commonly occurs for modules that implement
functions which have a multiplicity of possible solutions,
e.g., due to insufficient constraints or due to the need
to produce an approximate solution when not all 
constraints can be satisfied. If, because of habit or 
convenience, the solution is made unique by the
aribitrary application of an ancillary criterion
(e.g., least-squares) which does not 
ensure unit consistency then any system that
uses that module will be prevented from applying
unit consistency as a test of system integrity.

A goal of this paper is to promote consistency 
analysis, and in particular unit consistency,
as a consideration during all levels of component
and system development whenever applicable.
Unit consistency has been suggested in the past
as a consideration in specific applications (e.g.,
robotics~\cite{duffy90,doty2} and
data fusion~\cite{jku95}), but
to advocate it as a general design consideration
requires development of UC
alternatives to some of the most widely-used
mathematical tools in engineering.
For example, determining a ``best'' solution 
to an underdetermined or overdetermined 
set of equations can often be formulated -- either implicitly 
or explicitly -- in terms of a generalized matrix inverse.
However, the most commonly applied inverse 
(Moore-Penrose~\cite{moore, penrose,big})
is {\em not} unit consistent\footnote{Concern 
about unit inconsistencies resulting from 
naive applications of the Moore-Penrose inverse
(i.e., least-squares solutions) was raised most
vocally by Doty~\cite{doty1,doty2} in the 
early 1990s in the context of hybrid-control
systems.  He referred to the problem as having
``only been discussed in back hallways at 
conferences or in private meetings'' prior to
the work of Duffy~\cite{doty1,duffy90}.} . In fact, 
the most commonly applied tools in linear systems analysis,
the eigen and singular-value decompositions, are inherently
not unit consistent and therefore require
UC alternatives. 

The structure of the paper is as follows:
Section~\ref{gmisec} motivates the importance
of preserving salient system properties as a means
for determining unique solutions to ill-posed inverse
problems and discusses the requirements for a 
unit-consistent generalized inverse. Section~\ref{lrucinv} 
develops left and right unit-consistent generalized 
inverses. Section~\ref{nzginv} develops a unit-consistent 
generalized inverse for elemental nonzero matrices,
and Section~\ref{gensec} develops a fully general
unit-consistent generalized matrix inverse.
Section~\ref{application} provides an example
of the use of consistency analysis to the problem of
obtaining a linearized approximation to an unknown
nonlinear transformation based on a set of input and
output vectors.
Section~\ref{ucsvd} applies the techniques used
to achieve unit consistency for the generalized inverse 
problem to develop unit-consistent and unit-invariant
alternatives to the singular value decomposition (SVD) 
and other tools from linear algebra. 
Section~\ref{application2} provides an example
of the use of unit-invariant singular values to 
provide retrieval robustness to multiplicative noise in
an image database application. Finally, Section~\ref{discsec} 
summarizes and discusses the contributions of the paper.

\section{Generalized Matrix Inverses}
\label{gmisec}

For a nonsingular $n\times n$  matrix\footnote{It is assumed throughout that 
matrices are defined over an associative normed division algebra (real, complex, 
and quaternion). The inverse of a unitary matrix $\Um$ can therefore be
expressed using the conjugate-transpose operator as $\Umi$.}  
$\Amn$ there exists a unique matrix inverse,
$\Amni$, for which certain properties of scalar inverses 
are preserved, e.g., commutativity:
\begin{equation}
      \Amn\Amni ~ = ~ \Amni\Amn ~ = ~ \Imat
\end{equation}
while others have direct analogs, e.g., matrix inversion distributes
over nonsingular multiplicands as:
\begin{equation}
      \inv{(\Xm\Amn\Ym)} ~ = ~ \Ymi\Amni\Xmi 
\end{equation}
When attempting to generalize the notion of a matrix
inverse for singular $\Amn$ it is only possible to 
define an approximate inverse, $\Atinv$, which retains
a subset of the algebraic properties of a true
matrix inverse. For example, 
a generalized inverse definition might simply require
the product $\Amn\Atinv$ to be idempotent in analogy to the identity matrix. 
Alternative definitions might further require:
\begin{equation}
      \Amn\Atinv\Amn ~ = ~ \Amn 
\end{equation}
or
\begin{equation}
      \Atinv\Amn ~ = ~ \Amn\Atinv
\end{equation}
and/or other properties that may be of analytic
or application-specific utility. 

The vast literature\footnote{Much of this literature is covered
in the comprehensive book by Ben-Israel and Greville~\cite{big}.}
on generalized inverse theory spans more than a century and
can inform the decision about which of the many 
possible generalized inverses is best suited to the needs of a
particular application.
For example, the {\em Drazin} inverse, $\Adinv$, satisfies the following
for any square matrix $\Amn$ and nonsingular matrix 
$\Xm$~\cite{drazin,cmr76,big}:
\begin{eqnarray}
    & ~ &       \Adinv\Amn\Adinv ~ = ~ \Adinv \\
    & ~ &       \Amn\Adinv ~ = ~ \Adinv\Amn \label{dinvcom} \\
    & ~ &       \dinv{(\Xm\Amn\Xmi)} ~ = ~ \Xm\Adinv\Xmi \label{dinvcnst}
\end{eqnarray}
Thus it is applicable when there is need for
commutativity (Eq.(\ref{dinvcom})) and/or consistency with respect
to similarity transformations (Eq.(\ref{dinvcnst})). On the
other hand, the Drazin inverse is only defined for square matrices
and does not guarantee the rank of 
$\Adinv$ to be the same as $\Amn$. Because
the rank of $\Adinv$ may be less than that of $\Amn$ 
(and in fact is zero for all nilpotent matrices), it is
not appropriate for recursive control and estimation
problems (and many other 
applications) that cannot accommodate
progressive rank reduction. 

The Moore-Penrose {\em pseudoinverse}, $\Apinv$, is defined 
for any $m\times n$ matrix $\Amn$ and satisfies conditions
which include the following for any conformant unitary matrices 
$\Um$ and $\Vm$:
\begin{eqnarray}
    & ~ &      \mbox{rank}[\Apinv] ~ = ~ \mbox{rank}[\Amn] \\
    & ~ &      \Amn\Apinv\Amn ~ = ~ \Amn \\
    & ~ &       \Apinv\Amn\Apinv ~ = ~ \Apinv \\
    & ~ &       \pinv{(\Um\Amn\Vm)} ~ = ~ \Vmi\Apinv\Umi \label{udist}
\end{eqnarray}
Its use is therefore appropriate when there is need for
unitary consistency, i.e., as guaranteed by Eq.(\ref{udist}). 
Despite its near-universal use throughout many areas of 
science and engineering ranging from tomography~\cite{berryman00} 
to genomics analysis~\cite{altergolub04}, the Moore-Penrose
inverse is not appropriate for many problems to which it is
commonly applied, e.g., state-space applications 
that require consistency with respect to the choice of units
for state variables. In the case of a square singular transformation
matrix $\Amn$, for example, a simple change of units applied 
to a set of state variables may require an
inverse $\Atinv$ to be perserved under diagonal similarity
\begin{equation}
      \tinv{(\Dm\Amn\Dmi)} ~ = ~ \Dm\Atinv\Dmi
\end{equation}  
where the diagonal matrix $\Dm$ defines an arbitary 
change of units. The Moore-Penrose inverse does not
satisfy this requirement because $\pinv{(\Dm\Amn\Dmi)}$
does not generally equal $\Dm\Apinv\Dmi$. As a concrete
example, given
\begin{equation}
\label{beginexample}
\Dm ~=~ \left[\begin{array}{cc}
                         1 & 0\\ 0 & 2 \end{array}\right] ~ ~ ~ ~
\Amn ~=~ \left[\begin{array}{cc}
                         1/2 & -1/2\\ 1/2 & -1/2 \end{array}\right] ~ ~ ~ ~
\end{equation}
it can be verified that 
\begin{equation}
\Apinv ~=~ \left[\begin{array}{cc}
                         ~~1/2 & ~~1/2\\ -1/2 & -1/2 \end{array}\right] 
\end{equation}
and that 
\begin{equation}
\Dm\Apinv\Dmi ~=~ \left[\begin{array}{cc}
                                         ~1/2 & ~1/4\\ -1 & -1/2 \end{array}\right] 
\end{equation}
which does not equal
\begin{equation}
\label{endexample}
\pinv{(\Dm\Amn\Dmi)} ~=~ \left[\begin{array}{cc}
                                                     ~~0.32 & ~~0.64\\ -0.16 & -0.32 \end{array}\right]. 
\end{equation}

To appreciate the significance of
unit consistency, consider the standard
linear model
\begin{equation}
     \yv ~ = ~ \Amn\cdot\av \label{linmod}
\end{equation}
where the objective is to identify a vector $\av$ of parameter values
satisfying the above equation for a data matrix $\Amn$ and
a known/desired state vector $\yv$. If $\Amn$ is nonsingular
then there exists a unique $\Amni$ which gives the solution
\begin{equation}
     \av ~ = ~ \Amni\cdot\yv
\end{equation}
If, however, $\Amn$ is singular then the
Moore-Penrose inverse could be applied as
\begin{equation}
     \av ~ = ~ \Apinv\cdot\yv
\end{equation}
to obtain a solution. Now 
suppose that $\yv$ and $\av$ are expressed in 
different units as
\begin{eqnarray}
     \yyv & = & \Dm\yv \\
     \aav & = & \Em\av
\end{eqnarray}
where the diagonal matrices $\Dm$ and $\Em$ represent changes
of units, 
e.g.,  from imperial to metric, or rate-of-increase in liters-per-hour 
to a rate-of-decrease in liters-per-minute, or
any other multiplicative change of units. Then Eq.(\ref{linmod}) can 
be rewritten in the new units as
\begin{equation}
     \yyv ~ = ~ \Dm\yv  ~ = ~ (\Dm\Amn\Emi)\cdot\Em\av
     ~ = ~ (\Dm\Amn\Emi)\cdot\aav
\end{equation}
but for which  
\begin{equation}
      \Em\av   ~ \neq ~ \pinv{(\Dm\Amn\Emi)}\cdot\yyv 
\end{equation}
In other words, the change of units applied to the input
does not generally produce the same output in the new units.
This is because the Moore-Penrose inverse only 
guarantees consistency with 
respect to unitary transformations and not with
respect to nonsingular diagonal transformations.
To ensure unit consistency in this example a generalized matrix 
inverse $\Atinv$ would have to satisfy 
\begin{eqnarray}
      \tinv{(\Dm\Amn\Emi)}  & = & \Em\Atinv\Dmi
\end{eqnarray}
Stated more generally, if $\Amn$ represents a mapping
$V\rightarrow W$ from a vector space $V$ to a vector space
$W$ then the inverse transformation $\Atinv$ must preserve
consistency with respect to the application of arbitrary 
changes of units to the coordinates (state variables) associated
with $V$ and $W$.

Given a singular $\Amn$ it is not possible to say that a 
solution obtained using one definition of a generalized inverse 
is generally ``better'' or ``worse'' than that from any
other, but there are certainly application-specific 
considerations that can govern the choice. As has been
discussed, maintenance of unit consistency in all 
components of a system permits the higher-level
integrity of that system to be sanity-checked 
for unit consistency. An equally valuable practical
benefit is that a fully unit-consistent system that has
been extensively tuned and validated is guaranteed
to produce entirely predictable results if a change of
units is applied. By contrast, if multiple modules violate unit 
consistency in various ways, e.g., through use of the
Moore-Penrose inverse, then the effect of a subsequent
change of units on system performance may be
highly unpredictable. This motivates the derivation
of unit-consistent generalized matrix inverses and,
more importantly, a general methodology for 
achieving unit consistency.

\section{Left and Right UC Generalized Inverses}
\label{lrucinv}

Inverse consistency with respect to a nonsingular left diagonal
transformation, $\linv{(\Dm\Amn)}=\linv{\Amn}\Dmi$, or a right nonsingular diagonal
transformation, $\rinv{(\Amn\Dm)}=\Dmi\rinv{\Amn}$, is straightforward to obtain.
The solution has likely been exploited implicitly in one form or another in many
applications over the years; however, its formal derivation and
analysis is a useful exercise to establish concepts and notation
that will be used later to derive the fully-general UC solution.

\begin{definition}
\label{defdl}
Given an $m\times n$ matrix $\Amn$, a {\em left diagonal scale function}, 
$\DL[\Amn]\in\mathbb{R}_+^{m\times m}$, is defined as giving
a positive diagonal matrix satisfying the following 
for all conformant positive diagonal matrices $\Dmp$, unitary diagonals $\Dmu$, 
permutations $\Perm$, and unitaries $\Um$:
\begin{eqnarray}
   & ~ & \DL[\Dmp\Amn]\cdot(\Dmp\Amn)   ~ = ~ \DL[\Amn]\cdot\Amn , \label{lscalinv}\\
   & ~ & \DL[\Dmu\Amn] ~ = ~ \DL[\Amn], \label{lduinv}\\
   & ~ & \DL[\Perm\Amn] ~ = ~ \Perm\cdot \DL[\Amn], \label{lperminv}\\
   & ~ & \DL[\Amn\Um] ~ = ~ \DL[\Amn] \label{lruinv} 
\end{eqnarray}
In other words, the product $\DL[\Amn]\cdot\Amn$ is invariant with respect to any positive left-diagonal
scaling of $\Amn$, and $\DL[\Amn]$ is consistent with respect to any left-permutation of $\Amn$ and is  invariant 
with respect to left-multiplication by any diagonal unitary and/or 
right-multiplication by any general unitary.
\end{definition}

\begin{lemma}
\label{dconst}
Existence of a left-diagonal scale function according to Definition~\ref{defdl} is established
by instantiating $\DL[\Amn]=\Dm$ with
\begin{equation}
   \Dm(i,i) ~ \doteq ~  
   \begin{cases}
              1/\norm{\Amn(i,:)}  & \norm{\Amn(i,:)}>0\\
             1 & \textnormal{otherwise}
   \end{cases}
\end{equation}
where $\Amn(i,:)$ is row $i$ of $\Amn$ and $\norm{\cdot}$  is a fixed unitary-invariant vector 
norm\footnote{The unitary-invariant norm used here is necessary only because of 
the imposed right-invariant condition of Eq.(\ref{lruinv}).}.
\end{lemma}

\begin{proof}
$\DL[\cdot]$ as defined by Lemma~\ref{dconst} is a
strictly positive diagonal as required, and the left scale-invariance
condition of Eq.(\ref{lscalinv}) holds trivially for any row of $\Amn$ with
all elements equal to zero and holds for every nonzero row $i$ by homogeneity 
for any choice of vector norm as 
\begin{eqnarray}
    \Dmp(i,i)\Amn(i,:) ~ / ~ \norm{\Dmp(i,i)\Amn(i,:)} & = & \Dmp(i,i)\Amn(i,:) ~ / ~ 
                                                             \left(\Dmp(i,i)\cdot\norm{\Amn(i,:)}\right)  \\
   ~ & = & \Amn(i,:) ~ / ~ \norm{\Amn(i,:)}. 
\end{eqnarray}

The left diagonal-unitary-invariance condition of Eq.(\ref{lduinv}) 
is satisfied as $|\Dmu(i,i)|=1$ implies
$|(\Dmu\Amn)(i,j)|=|\Amn(i,j)|$ for every element 
$j$ of row $i$ of $\Dmu\Amn$.
The left permutation-invariance of Eq.(\ref{lperminv}) holds as
element $\Dm(i,i)$ is indexed with respect to the rows of $\Amn$,
and the right unitary-invariance
condition of Eq.(\ref{lruinv}) is satisfied by 
the assumed unitary invariance of the vector norm applied
to the rows of $\Amn$. 
\end{proof}

If $\Amn$ has {\em full support}, i.e., no row or column with all
elements equal to zero, then $\DL[\Dmp\Amn]=\Dmpi\cdot\DL[\Amn]$.
If, however, there exists a row $i$ of $\Amn$ with all
elements equal to zero then the $i$th diagonal element of 
$\DL[\Dmp\Amn]$ is $1$ according to Lemma~\ref{dconst},
so the corresponding element of $\Dmpi\cdot\DL[\Amn]$ will
be different unless $\Dmpi(i,i)=1$. Eq.(\ref{lscalinv}) holds 
because such elements are only applied to scale
rows of $\Amn$ with all elements equal to zero. The following
similarly holds in general
\begin{equation}
     \DL[\Dmp\Amn]\cdot\Amn ~ = ~ \Dmpi\cdot\DL[\Amn]\cdot\Amn, \label{dlident}
\end{equation}
and because any row $i$ of zeros in $\Amn$ implies that column $i$ of $\pinv{\Amn}$ will be 
zeros, the following also holds in general
\begin{equation}
     \pinv{\Amn}\cdot\DL[\Dmp\Amn] ~ = ~ \pinv{\Amn}\cdot\Dmpi\cdot\DL[\Amn].
\end{equation}

At this point it is possible to
derive a left generalized inverse of an arbitrary $m \times n$ matrix
$\Amn$, denoted $\linv{\Amn}$, that is consistent with 
respect to multiplication on the left by an arbitrary nonsingular diagonal matrix.

\begin{theorem}
\label{linvt}
For $m\times n$ matrix $\Amn$, the operator 
\begin{equation}
   \linv{\Amn} ~ \doteq ~ \pinv{\left(\DL[\Amn]\cdot\Amn\right)}\cdot\DL[\Amn]
\end{equation}
satisfies for any nonsingular diagonal matrix $\Dm$:
\begin{eqnarray}
   \Amn\linv{\Amn}\Amn & = & \Amn, \\
   \linv{\Amn}\Amn\linv{\Amn} & = & \linv{\Amn}, \\
   \linv{(\Dm\Amn)} & = & \linv{\Amn} \Dmi, \\
   \mbox{rank}[\linv{\Amn}] & = & \mbox{rank}[\Amn] 
\end{eqnarray}
and is therefore a left unit-consistent generalized inverse.
\end{theorem}
\begin{proof}
The first two generalized inverse properties can be established from the corresponding properties
of the Moore-Penrose inverse as: 
\begin{eqnarray}
\Amn\linv{\Amn}\Amn & = & \Amn \cdot \left\{\pinv{\left(\DL[\Amn]\cdot\Amn\right)}\cdot\DL[\Amn]\right\} \cdot \Amn \\
~ & = & \Amn \cdot \left\{\pinv{\left(\DL[\Amn]\cdot\Amn\right)}\cdot \left(\DL[\Amn]\cdot \Amn\right)\right\} \\
~ & = & (\inv{\DL[\Amn]}\cdot\DL[\Amn])\cdot \Amn \cdot \left\{\pinv{\left(\DL[\Amn]\cdot\Amn\right)}\cdot 
              \left(\DL[\Amn]\cdot \Amn\right)\right\} \\
~ & = & \inv{\DL[\Amn]}\cdot \left\{\underline{\left(\DL[\Amn]\cdot\Amn\right) \cdot \pinv{\left(\DL[\Amn]\cdot\Amn\right)}\cdot 
              \left(\DL[\Amn]\cdot \Amn\right)}\right\} \\
~ & = & \inv{\DL[\Amn]}\cdot \left(\DL[\Amn]\cdot \Amn\right)\\
~ & = & \Amn
\end{eqnarray}
and
\begin{eqnarray}
\linv{\Amn}\Amn\linv{\Amn} & = & \left\{\pinv{\left(\DL[\Amn]\cdot\Amn\right)}\cdot\DL[\Amn]\right\}\cdot\Amn\cdot
                                                        \left\{\pinv{\left(\inv{\DL[\Amn]}\Amn\right)}\cdot\DL[\Amn]\right\}  \\
~ & = & \left\{\underline{\pinv{\left(\DL[\Amn]\cdot\Amn\right)}\cdot\left(\DL[\Amn]\cdot\Amn\right)\cdot
                                                        \pinv{\left(\DL[\Amn]\cdot\Amn\right)}}\right\}\cdot \DL[\Amn]  \\
~ & = &  \pinv{\left(\DL[\Amn]\cdot\Amn\right)} \cdot \DL[\Amn] \\
~ & = & \linv{\Amn}
\end{eqnarray}
The left unit-consistency condition $\linv{(\Dm\Amn)}=\linv{\Amn}\Dmi$, for any nonsingular diagonal
matrix $\Dm$, can be established using a polar 
decomposition $\Dm=\Dmp\Dmu$:
\begin{eqnarray}
   \Dmp & = & \Abs[\Dm] \\
   \Dmu & = &  \Dm\Dmpi
\end{eqnarray}
and exploiting unitary-consistency of the Moore-Penrose 
inverse, i.e., $\pinv{(\Um\Amn)}=\pinv{\Amn}\Umi$, and 
commutativity of $\DL[\cdot]$ with other diagonal matrices:

\begin{eqnarray}
\linv{(\Dm\Amn)} & = &  \pinv{\left(\DL[\Dm\Amn]\cdot\Dm\Amn\right)}\cdot\DL[\Dm\Amn] \\
~ & = & \pinv{\left(\DL[\Dmp\Dmu\Amn]\cdot\Dmp\Dmu\Amn\right)}\cdot\DL[\Dmp\Dmu\Amn] \\
~ & = & \pinv{\left(\DL[\Amn]\cdot\Dmpi\cdot\Dmp\Dmu\Amn\right)}\cdot\Dmpi\cdot\DL[\Amn]  \\
~ & = & \pinv{\left(\DL[\Amn]\cdot(\Dmpi\Dmp)\cdot\Dmu\Amn\right)}\cdot\Dmpi\cdot\DL[\Amn]  \\
~ & = & \pinv{\left(\DL[\Amn]\cdot\Dmu\Amn\right)}\cdot\Dmpi\cdot\DL[\Amn] \label{lmpua}\\
~ & = & \pinv{\left(\DL[\Amn]\cdot\Amn\right)}\cdot\Dmui\cdot\Dmpi\cdot\DL[\Amn] \label{lmpub}\\
~ & = & \pinv{\left(\DL[\Amn]\cdot\Amn\right)}\cdot(\Dmui\Dmpi)\cdot\DL[\Amn] \\
~ & = & \pinv{\left(\DL[\Amn]\cdot\Amn\right)}\cdot\Dmi\cdot\DL[\Amn] \\
~ & = & \left\{\pinv{\left(\DL[\Amn]\cdot\Amn\right)}\cdot\DL[\Amn]\right\}\cdot\Dmi\\
~ & = & \linv{\Amn} \Dmi .
\end{eqnarray}
Lastly, the rank-consistency condition, $\mbox{rank}[\linv{\Amn}] =  \mbox{rank}[\Amn]$,
is satisfied as every operation performed according to
Lemma~\ref{dconst} preserves the rank of the original matrix.
In particular, the rank consistency of $\linv{\Amn}$ derives from the fact
that $\mbox{rank}[\Apinv]  =  \mbox{rank}[\Amn]$.
\end{proof}

A right unit-consistent generalized inverse clearly can be 
derived analogously or  in terms
of the already-defined left operator as
\begin{equation}
   \rinv{\Amn} ~ \doteq ~ \left(\linv{(\Amn^{\mbox{\tiny T}})}\right)^{\mbox{\tiny T}}.
\end{equation}
In terms of the linear model of Eq.(\ref{linmod}) for determining values
for parameters $\av$,
\begin{eqnarray}
     \yv & = & \Amn\cdot\av \nonumber \\
       ~ & \Downarrow & ~ \nonumber \\
     \av & = & \Atinv\cdot\yv \nonumber
\end{eqnarray}
the inverse $\Atinv$ could be instantiated with 
either $\linv{\Amn}$ or $\rinv{\Amn}$ to
provide, respectively, consistency with respect to
the application of a change of units to $\yv$ or
a change of units to $\av$ -- {\em but not both}. 

\section{UC Generalized Inverse for Elemental-Nonzero Matrices}
\label{nzginv}

The derivations of separate left and right UC inverses 
from the previous section cannot
be applied to achieve general unit consistency, i.e., to obtain
a UC generalized inverse
$\Aginv$ which satisfies
\begin{equation}
   \ginv{(\Dm\Amn\Em)} ~ = ~ \Emi  \Aginv \Dmi
\end{equation}
for arbitrary nonsingular diagonals $\Dm$ and $\Em$. However, a 
{\em joint} characterization of the left and right diagonal transformations
can provide a basis for doing so.

\begin{lemma}
\label{daex}
The transformation of an $m\times m$ matrix $\Amn$ as $\Dm\Amn\Em$, with $m\times m$
diagonal $\Dm$ and $n\times n$ diagonal $\Em$, is equivalent to a Hadamard
(elementwise) matrix product $\Xm\circ\Amn$ for some rank-1 matrix $\Xm$.
\end{lemma}

\begin{proof}
Letting $\dmc=\Diag[\Dm]$ and $\enc=\Diag[\Em]$, the matrix product $\Dm\Amn\Em$ can be expressed as
\begin{eqnarray}
   \Dm\Amn\Em & = & (\dmc\onr)\circ\Amn\circ(\omc\emr) \\
   ~                    & = & \left\{(\dmc\onr)\circ(\omc\emr)\right\}\circ\Amn \\ 
   ~                   & = & (\dmc\enr)\circ\Amn
\end{eqnarray}
where $\onr$ is a row vector of $n$ ones and
$\omc$ is a column vector of $m$ ones.
Letting $\Xm=\dmc\enr$ completes the proof.
\end{proof}

\begin{definition}
\label{defdg}
For an $m\times n$ matrix $\Amn$, left and right {\em general-diagonal scale functions} 
$\DGL[\Amn]\in\mathbb{R}_+^{m\times m}$ and
$\DGR[\Amn]\in\mathbb{R}_+^{n\times n}$
are defined as jointly satisfying the following 
for all conformant positive diagonal matrices $\Dmp$ and $\Emp$, unitary diagonals $\Dmu$ and $\Dmv$, 
and permutations $\Perm$ and $\Qm$:
\begin{eqnarray}
   & ~ & \DGL[\Dmp\Amn\Emp]\cdot(\Dmp\Amn\Emp)\cdot\DGR[\Dmp\Amn\Emp]
                      ~ = ~ \DGL[\Amn]\cdot\Amn\cdot\DGR[\Amn] \label{defdg1}\\
   & ~ & \DGL[\Perm\Amn\Qm]\cdot(\Perm\Amn\Qm)\cdot\DGR[\Perm\Amn\Qm]
                      ~ = ~ \Perm\cdot\left\{\DGL[\Amn]\cdot\Amn\cdot\DGR[\Amn]\right\}\cdot\Qm \label{defdg2}\\
   & ~ & \DGL[\Dmu\Amn\Dmv] ~ = ~ \DGL[\Amn], \label{defdg3}\\
   & ~ & \DGR[\Dmu\Amn\Dmv] ~ = ~ \DGR[\Amn] \label{defdg4}
\end{eqnarray}
The function $\DG[\Amn]$ is defined to be the rank-1 matrix guaranteed by
Lemma~\ref{daex} 
\begin{equation}
   \DG[\Amn] \circ \Amn ~ \equiv ~ \DGL[\Amn]\cdot\Amn\cdot\DGR[\Amn]
\end{equation}
i.e.,
\begin{equation}
   \DG[\Amn] ~ = ~ \Diag[\DGL[\Amn]]\cdot\Diag[\DGR[\Amn]]^{\mbox{\tiny T}}
\end{equation}
\end{definition}

\begin{definition}
A matrix $\Amn$ is defined to be an {\em elemental-nonzero matrix} if and only if
it does not have any element equal to zero. 
\end{definition}

The following lemma uses the elementwise matrix functions $\Log[\cdot]$ and $\Exp[\cdot]$,
where $\Log[\Amn]$ represents the result of taking the logarithm of the magnitude
of each element of $\Amn$ and $\Exp[\Amn]$ represents the taking of the
exponential of every element of $\Amn$.

\begin{lemma}
\label{nonzerodef}
Existence of a general diagonal scale function according to Definition~\ref{defdg} for
arguments without zero elements is established
by instantiating $\DGL[\Amn]$ and $\DGR[\Amn]$ as
\begin{eqnarray}
   \DGL[\Amn] & = & \Diag[\xmc] \\
   \DGR[\Amn] & = & \Diag[\ync] \\
   ~ & \textnormal{for} & ~ \nonumber \\
   \xmc\cdot\ynr & = &  \DG[\Amn] ~ = ~ \Exp\left[ \Jm\Lmn\Jn - \Lmn\Jn - \Jm\Lmn  \right] 
\end{eqnarray}
where $\Lmn=\Log[\Amn]$ and $\Jm$ has all elements equal to $1/m$ and $\Jn$ has all 
elements equal to $1/n$.
\end{lemma}

\begin{proof}
First it must be shown that $\Exp\left[ \Jm\Lmn\Jn - \Lmn\Jn - \Jm\Lmn  \right]$ is a rank-1 matrix. 
This can be achieved by expanding as
\begin{eqnarray}
       \Jm\Lmn\Jn - \Lmn\Jn - \Jm\Lmn  & = &  
                       (\frac{1}{2}\Jm\Lmn\Jn - \Lmn\Jn) +  (\frac{1}{2}\Jm\Lmn\Jn - \Jm\Lmn)\\
      ~ & = &   (\frac{1}{2}\Jm\Lmn - \Lmn) \Jn + \Jm(\frac{1}{2}\Lmn\Jn - \Lmn)\\
      ~ & = &  \umc\onr + \omc\vnr 
\end{eqnarray}
where 
\begin{eqnarray}   
     \umc & = & \frac{1}{n}\left(\frac{1}{2}\Jm - \Imm\right)\Lmn \cdot \onc \\
     \vnr & = &  \omr \cdot \Lmn\left(\frac{1}{2}\Jn - \Imn\right)/m
\end{eqnarray}
and then noting that the elementwise exponential of  $\umc\onr + \omc\vnr$ is the strictly positive
rank-1 matrix $\Exp[\umc]\cdot\Exp[\vnr]$, i.e., $\DGL[\Amn]=\Diag[\Exp[\umc]]$ and
$\DGR[\Amn]=\Diag[\Exp[\vmc]]$, which confirms existence and strict positivity as required.
Eq.(\ref{defdg1}) can be established by observing that
\begin{equation}
   \DGL[\Dmp\Amn\Emp]\cdot(\Dmp\Amn\Emp)\cdot\DGR[\Dmp\Amn\Emp] ~ \equiv ~ \DG[\Dmp\Amn\Emp]\circ(\Dmp\Amn\Emp)
\end{equation}
and letting 
$\dmc=\Diag[\Dmp]$,
$\enc=\Diag[\Emp]$, $\umc=\Log[\dmc]$,
$\vnc=\Log[\emc]$, and $\Lmn=\Log[\Amn]$:
%\begin{singlespacing}
\begin{eqnarray}   
    \DG[\Dmp\Amn\Emp] & = & \DG[ ~ (\dmc\onr) \circ \Amn \circ (\omc\enr) ~ ] \\
    ~ & = &  \DG[~ \Exp[\umc\onr]\circ \Amn \circ \Exp[\omc\vnr] ~ ] \\
    ~ & = & \Exp[~ \Jm(\umc\onr + \Lmn+ \omc\vnr)\Jn\\
    ~ & ~ &  ~ \phantom{\Exp[} - (\umc\onr + \Lmn+ \omc\vnr)\Jn  \\
    ~ & ~ &  ~ \phantom{\Exp[}  - \Jm(\umc\onr + \Lmn+ \omc\vnr) ~]\\
    ~ & = & \Exp[~ (\Jm\cdot\umc\onr + \Jm\Lmn\Jn + \omc\vnr\cdot\Jn) \\
    ~ & ~ &  ~ \phantom{\Exp[} - (\umc\onr + \Lmn\Jn + \omc\vnr\cdot\Jn) \\
    ~ & ~ &  ~ \phantom{\Exp[}  -  (\Jm\cdot\umc\onr + \Jm\Lmn + \omc\vnr) ~] \\
    ~ & = & \Exp[~ (-\umc\onr) + (\Jm\Lmn\Jn-\Lmn\Jn-\Jm\Lmn) + (-\omc\vnr) ~] \\
    ~ & = & \Exp[-\umc\onr] ~ \circ ~ \DG[\Amn]  ~ \circ ~ \Exp[- \omc\vnr] \\
    ~ & = & \Dmpi\cdot\DG[\Amn]\cdot\Empi  \label{disei}
\end{eqnarray}
%\end{singlespacing}
where the last step recognizes that $-\umc=\Log[\Diag[\Dmpi]]$ and
$-\vnc=\Log[\Diag[\Empi]]$. The identity of Eq.(\ref{defdg1}) can then be shown as:
\begin{eqnarray}
      \DGL[\Dmp\Amn\Emp]\cdot(\Dmp\Amn\Emp)\cdot\DGR[\Dmp\Amn\Emp] & = & \DG[\Dmp\Amn\Emp]\circ(\Dmp\Amn\Emp) \\
      ~ & = & (\Dmpi\cdot\DG[\Amn]\cdot\Empi) \circ (\Dmp\Amn\Emp) \\
      ~ & = & \DG[\Amn]\circ\Amn \\
      ~ & = & \DGL[\Amn]\cdot\Amn\cdot\DGR[\Amn]
\end{eqnarray}
Eq.(\ref{defdg2})  holds as the indexing of the rows and columns of 
$\DGL[\Amn]$ and $\DGR[\Amn]$ (and $\DG[\Amn]$) is the same as that of $\Amn$.
Eqs.(\ref{defdg3}) and~(\ref{defdg4}) hold directly because Lemma~\ref{nonzerodef} 
only involves functions of the
absolute values of the elements of the argument matrix $\Amn$.
\end{proof}

\begin{theorem}
\label{ginvtp}
For an elemental-nonzero $m\times n$ matrix $\Amn$, the operator 
\begin{equation}
   \Aginv ~ \doteq ~ \DGR[\Amn]\cdot
         \pinv{\left(\DGL[\Amn]\cdot\Amn\cdot\DGR[\Amn]\right)}\cdot\DGL[\Amn]
         \label{ginvdef}
\end{equation}
satisfies for any nonsingular diagonal matrices $\Dm$ and $\Em$:
\begin{eqnarray}
   \Amn\Aginv\Amn & = & \Amn, \\
   \Aginv\Amn\Aginv & = & \Aginv, \\
   \ginv{(\Dm\Amn\Em)} & = & \Emi\Aginv\Dmi, \label{ginvuc}\\
   \mbox{rank}[\Aginv] & = & \mbox{rank}[\Amn] 
\end{eqnarray}
and is therefore a general unit-consistent generalized inverse.
\end{theorem}

\begin{proof}
The first two generalized inverse properties can be established from the corresponding properties
of the MP-inverse as: 
\begin{eqnarray}
\Amn\Aginv\Amn & = & \Amn\cdot \left\{\DGR[\Amn]\cdot
          \pinv{\left(\DGL[\Amn]\cdot\Amn\cdot\DGR[\Amn]\right)} 
          \cdot\DGL[\Amn]\right\} \cdot \Amn \\
~ & = & (\inv{\DGL[\Amn]}\cdot\DGL[\Amn])\cdot \\
~ & ~ & \phantom{AAA} \Amn\cdot\left\{\DGR[\Amn]
                         \cdot\pinv{\left(\DGL[\Amn]\cdot\Amn\cdot\DGR[\Amn]\right)}
                         \cdot\DGL[\Amn]\right\}\cdot\Amn \\
 ~ & ~ &  \phantom{AAAAAA} \cdot(\DGR[\Amn]\cdot\inv{\DGR[\Amn]}) \\
~ & = & \inv{\DGL[\Amn]}\cdot \\
 ~ & ~ &  \phantom{A} \underline{(\DGL[\Amn]\cdot\Amn\cdot\DGR[\Amn])
          \cdot\pinv{\left(\DGL[\Amn]\cdot\Amn\cdot\DGR[\Amn]\right)}\cdot
          (\DGL[\Amn]\cdot\Amn\cdot\DGR[\Amn])} \\
 ~ & ~ &  \phantom{AA} \cdot\inv{\DGR[\Amn]} \\
~ & = & \inv{\DGL[\Amn]}\cdot(\DGL[\Amn]\cdot\Amn\cdot\DGR[\Amn])
          \cdot\inv{\DGR[\Amn]} \\ 
~ & = & \Amn
\end{eqnarray}
and
\begin{eqnarray}
\Aginv\Amn\Aginv & = & \left\{\DGR[\Amn]\cdot
          \pinv{\left(\DGL[\Amn]\cdot\Amn\cdot\DGR[\Amn]\right)} 
          \cdot\DGL[\Amn]\right\}\cdot\Amn \\
 ~ & ~ &  ~ ~ ~ ~ ~ ~\cdot \left\{\DGR[\Amn]
          \cdot\pinv{\left(\DGL[\Amn]\cdot\Amn\cdot\DGR[\Amn]\right)}
          \cdot\DGL[\Amn]\right\} \\
~ & = & \DGR[\Amn]\\
 ~ & ~ &  \phantom{A}  \cdot \pinv{\left(\DGL[\Amn]\cdot\Amn\cdot\DGR[\Amn]\right)}\cdot 
          \underline{\left(\DGL[\Amn]\cdot\Amn\cdot\DGR[\Amn]\right)
          \cdot\pinv{\left(\DGL[\Amn]\cdot\Amn\cdot\DGR[\Amn]\right)}} \\
  ~ & ~ &  \phantom{AA} \cdot\DGL[\Amn] \\
~ & = & \DGR[\Amn]\cdot
        \pinv{\left(\DGL[\Amn]\cdot\Amn\cdot\DGR[\Amn]\right)}
        \cdot\DGL[\Amn] \\
~ & = & \Aginv
\end{eqnarray}

The general UC condition $\ginv{(\Dm\Amn\Em)}=\Emi  \Aginv \Dmi$, for any nonsingular diagonal
matrix $\Dm$, can be established using a polar decompositions $\Dm=\Dmp\Dmu$ and $\Em=\Emp\Emu$:
\begin{eqnarray}
\ginv{(\Dm\Amn\Em)} & = &  \DGR[\Dm\Amn\Em]\cdot
         \pinv{\left(\DGL[\Dm\Amn\Em]\cdot(\Dm\Amn\Em)\cdot\DGR[\Dm\Amn\Em]\right)}\cdot\DGL[\Dm\Amn\Em] \\
~ & = & \DGR[\Dmp\Amn\Emp]\cdot
         \pinv{\left(\DGL[\Dmp\Amn\Emp]\cdot(\Dm\Amn\Em)\cdot\DGR[\Dmp\Amn\Emp]\right)}\cdot\DGL[\Dmp\Amn\Emp] \\
~ & = & \Empi\cdot\DGR[\Amn]\cdot 
              \pinv{\left(\DGL[\Amn]\cdot\Dmpi\cdot(\Dm\Amn\Em)\cdot\Empi\cdot\DGR[\Amn]\right)}
              \cdot\DGL[\Amn]\cdot\Dmpi \\
~ & = &  \Empi\cdot\DGR[\Amn]\cdot 
              \pinv{\left(\DGL[\Amn]\cdot\Dmu\cdot\Amn\cdot\Emu\cdot\DGR[\Amn]\right)}
              \cdot\DGL[\Amn]\cdot\Dmpi     \label{gpinvua}  \\
~ & = & \Empi\cdot\DGR[\Amn]\cdot\Emui\cdot
              \pinv{\left(\DGL[\Amn]\cdot\Amn\cdot\DGR[\Amn]\right)}
              \cdot\Dmui\cdot\DGL[\Amn]\cdot\Dmpi    \label{gpinvub}   \\
~ & = & (\Empi\cdot\Emui)\cdot\DGR[\Amn]\cdot
              \pinv{\left(\DGL[\Amn]\cdot\Amn\cdot\DGR[\Amn]\right)}
              \cdot\DGL[\Amn]\cdot(\Dmui\cdot\Dmpi) \\
~ & = & \Emi\cdot\left\{\DGR[\Amn]\cdot
              \pinv{\left(\DGL[\Amn]\cdot\Amn\cdot\DGR[\Amn]\right)}
              \cdot\DGL[\Amn]\right\}\cdot\Dmi\\
~ & = & \Emi\cdot\Aginv\cdot\Dmi .
\end{eqnarray}
The rank-consistency condition of the theorem holds exactly as 
for the proof of Theorem~\ref{linvt}.
\end{proof}

The elemental-nonzero condition of Lemma~\ref{nonzerodef} is required
to ensure the existence of the elemental logarithms for $\Lmn=\Log[\Amn]$, so
the closed-form solution for the general unit-consistent matrix
inverse of Theorem~\ref{ginvtp} is applicable only to matrices without
zero elements. In many contexts involving general matrices there
is no reason to expect any elements to be identically zero, but in
some applications, e.g., compressive sensing, zeros are structurally
enforced. Unfortunately, Lemma~\ref{nonzerodef} cannot be extended
to accommodate zeros by a simple limiting strategy; however, 
results from matrix scaling theory can be applied to derive an 
unrestricted solution.

\section{The Fully-General Unit-Consistent Generalized Inverse}
\label{gensec}

Given a nonnegative matrix $\Amn\in\mathbb{R}^{m\times n}$ with 
full support, $m$ positive numbers $S_1...S_m$, and $n$ positive 
numbers $T_1...T_n$, Rothblum \& Zenios~\cite{rz92} investigated the problem
of identifying positive diagonal matrices $\Um\in\mathbb{R}^{m\times m}$ and 
$\Vm\in\mathbb{R}^{n\times n}$ such that the product of the 
nonzero elements of each row $i$ of $\Amn'=\Um\Amn\Vm$ is $S_i$ and
the product of the nonzero elements of each column $j$ of 
$\Amn'$ is $T_j$. They provided an efficient solution, referred
to in their paper as {\em Program II}, and analyzed 
its properties. Specifically, for vectors 
$\mu\in\mathbb{R}^m$ and $\eta\in\mathbb{R}^n$, defined
in their paper\footnote{As will be seen, the precise definitions of
$\mu$ and $\eta$ will prove irrelevant for purposes of this paper.}, 
they proved the following:

\begin{theorem}
\label{rztheorem} (Rothblum \& Zenios\footnote{This theorem combines results from
theorems~4.2 and~4.3 of~\cite{rz92}. The variables $\Um$ and $\Vm$ are used 
here for positive diagonal matrices purely for consistency with that paper despite 
their exclusive use elsewhere in this paper to refer to unitary matrices. Although 
not stated explicitly by Rothblum and Zenios, Program~II
is easily verified to be permutation consistent.})
The following are equivalent:
\begin{enumerate}
   \item Program II is feasible.
   \item Program II has an optimal solution.
   \item $\prod_{i=1}^{m}(S_i)^{\mu_i} ~ = ~ \prod_{j=1}^{n}(T_j)^{\eta_j}$.
\end{enumerate}
If a solution exists then the matrix $\Amn'=\Um\Amn\Vm$ is the unique positive
diagonal scaling of $\Amn$ for which the product of the nonzero elements of each row
$i$ is $S_i$ and the product of the nonzero elements of each column $j$ is $T_j$.
\end{theorem}

Although $\Amn'$ is unique in Theorem~\ref{rztheorem}, the diagonal scaling 
matrices $\Um$ and $\Vm$ may not be. The implications of this, and the question
of existence, are addressed by the following theorem.  

\begin{theorem}
\label{rzcor}
For any nonnegative matrix $\Amn\in\mathbb{R}^{m\times n}$ with full support
there exist positive diagonal matrices $\Um\in\mathbb{R}^{m\times m}$ 
and $\Vm\in\mathbb{R}^{n\times n}$ such that the product of the 
nonzero elements of each row $i$ and column $j$ of $\Xm=\Um\Amn\Vm$ is $1$,
and $\Xm=\Um\Amn\Vm$ is the unique positive diagonal scaling of $\Amn$ which
has this property. Furthermore, if there do exist distinct positive diagonal 
matrices $\Um_1$, $\Vm_1$, $\Um_2$, and $\Vm_2$ such that 
\begin{equation}
\Xm ~=~ \Um_1\Amn\Vm_1 ~=~ \Um_2\Amn\Vm_2
\end{equation}
then $\inv{\Vm}_1\Apinv\inv{\Um}_1$ = $\inv{\Vm}_2\Apinv\inv{\Um}_2$.
\end{theorem}

\begin{proof}
The existence (and uniqueness) of a solution for Program~II according to 
Theorem~\ref{rztheorem} is equivalent to
\begin{equation}
   \prod_{i=1}^{m}(S_i)^{\mu_i} ~ = ~ \prod_{j=1}^{n}(T_j)^{\eta_j}
\end{equation}
which holds unconditionally, i.e., independent of $\mu$ and $\eta$, for
the case in which every $S_i$ and $T_j$ is~1. Proof of the {\em Furthermore} 
statement is given in Appendix~\ref{facta}.
\end{proof}

\begin{lemma}
\label{gendef}
Given an $m\times n$ matrix $\Amn$, let $\Xm$ be the matrix formed by removing
every row and column of $\Abs[\Amn]$ for which all elements are equal to zero, and define
$r[i]$ to be the row of $\Xm$ corresponding to row $i$ of $\Amn$ and $c[j]$
to be the column of $\Xm$ corresponding to column $j$ of $\Amn$. Let $\Um$ and
$\Vm$ be the diagonal matrices guaranteed to exist from the application of Program~II 
to $\Xm$ according to Theorem~\ref{rzcor}. 
Existence of a general-diagonal scale function according to Definition~\ref{defdg} for
$\Amn$ is provided by instantiating $\DGL[\Amn]=\Dm$ and $\DGR[\Amn]=\Em$ where
\begin{eqnarray}
   \Dm(i,i) & = &  
   \begin{cases}
              \Um(r[i],r[i])  & \textnormal{row $i$ of $\Amn$ is not zero}\\
             1 & \textnormal{otherwise}
   \end{cases}
\\
~ & ~ & \nonumber
\\
   \Em(j,j) & = &  
   \begin{cases}
              \Vm(c[j],c[j])  & \textnormal{column $j$ of $\Amn$ is not zero}\\
             1 & \textnormal{otherwise}
   \end{cases}
\end{eqnarray}
\end{lemma}

\begin{proof}
In the case that $\Amn$ has full support so that $\Xm=\Abs[\Amn]$ then
Theorem~\ref{rztheorem} guarantees that 
$\DGL[\Xm]\cdot\Xm\cdot\DGR[\Xm]$ is the unique diagonal scaling of 
$\Xm$ such that the product of the nonzero elements of each row and
column is 1. Therefore, the scale-invariance condition of Eq.(\ref{defdg1}):
\begin{equation}
    \DGL[\Dmp\Xm\Emp]\cdot(\Dmp\Xm\Emp)\cdot\DGR[\Dmp\Xm\Emp] 
            ~ = ~ \DGL[\Xm]\cdot\Xm\cdot\DGR[\Xm]
\end{equation}
holds for any positive diagonals $\Dmp$ and $\Emp$
as required. For the case of general $\Amn$ the construction 
defined by Lemma~\ref{gendef} preserves uniqueness with respect to nonzero 
rows and columns of $\DGL[\Amn]\cdot\Abs[\Amn]\cdot\DGR[\Amn]$,
i.e., those which correspond to the rows and columns of $\Um\Xm\Vm$, 
by the guarantee of Theorem~\ref{rztheorem}, and any row or column with all 
elements equal to zero is inherently scale-invariant, so  Eq.(\ref{defdg1}) holds
unconditionally for the construction defined by Lemma~\ref{gendef}.
The remaining conditions (permutation consistency and invariance
with respect to unitary diagonals) hold equivalently to the proof
of Lemma~\ref{nonzerodef}.
\end{proof}

At this point it is possible to establish the existence of a fully-general,
unit-consistent, generalized matrix inverse.

\begin{theorem}
\label{genginv}
For an $m\times n$ matrix $\Amn$ there exists an operator 
\begin{equation}
   \Aginv ~ \doteq ~ \DGR[\Amn]\cdot
         \pinv{\left(\DGL[\Amn]\cdot\Amn\cdot\DGR[\Amn]\right)}\cdot\DGL[\Amn]
\end{equation}
which satisfies for any nonsingular diagonal matrices $\Dm$ and $\Em$:
\begin{eqnarray}
   \Amn\Aginv\Amn & = & \Amn, \label{g1}\\
   \Aginv\Amn\Aginv & = & \Aginv, \label{g2}\\
   \ginv{(\Dm\Amn\Em)} & = & \Emi\Aginv\Dmi, \label{g3}\\
   %\ginv{(\Dm\Amn\Dmi)} & = & \Dm\Aginv\Dmi, \label{g4}\\
   \mbox{rank}[\Aginv] & = & \mbox{rank}[\Amn] \label{g5}.
\end{eqnarray}
\end{theorem}

\begin{proof}
The proof of Theorem~\ref{ginvtp} applies unchanged to
Theorem~\ref{genginv} except that the
elemental-nonzero condition imposed by Lemma~\ref{nonzerodef}
is removed by use of Lemma~\ref{gendef}.
\end{proof}

For completeness, the example of Eqs.(\ref{beginexample}-\ref{endexample})
with
\begin{equation}
\Dm ~=~ \left[\begin{array}{cc}
                         1 & 0\\ 0 & 2 \end{array}\right] ~ ~ ~ ~
\Amn ~=~ \left[\begin{array}{cc}
                         1/2 & -1/2\\ 1/2 & -1/2 \end{array}\right] ~ ~ ~ ~
\end{equation}
can be revisted to verify that 
\begin{equation}
\ginv{(\Dm\Amn\Dmi)} ~=~ \Dm\Aginv\Dmi ~=~ 
                 \left[\begin{array}{cc}
                         1/2 & ~1/4\\ -1 & -1/2 \end{array}\right] 
\end{equation}
where $\ginv{(\Dm\Amn\Dmi)}=\Dm\Aginv\Dmi$ as
expected. Extending the example with
\begin{equation}
\Em ~=~ \left[\begin{array}{cc}
                         5 & ~0\\ 0 & -3 \end{array}\right]
\end{equation}
it can be verified that 
\begin{equation}
\ginv{(\Dm\Amn\Em)} ~=~ \Emi\Aginv\Dmi ~=~ 
                 \left[\begin{array}{cc}
                         1/10 & 1/20\\ 1/6 & 1/12 \end{array}\right] 
\end{equation}
with equality as expected.

In practice the state space of interest may comprise subsets of variables
having different assumed relationships. For example, assume that 
$m$ state variables have incommensurate units while the remaining $n$
state variables are defined in a common Euclidean space, i.e., their relationship
should be preserved under orthogonal transformations. This assumption
requires that a linear transformation $\Amn$ should be consistent 
with respect to state-space transformations of the form
\begin{equation}
   \TR ~=~
   \left[\begin{array}{cc}
            \Dm   & {\bf 0} \\  
         {\bf 0} &    \Rm
   \end{array}\right]
\end{equation}
where $\Dm$ is a nonsingular $m\times m$ diagonal matrix and $\Rm$ is an
$n\times n$ orthogonal matrix. Thus the inverse of $\Amn$ cannot be 
obtained by applying 
either the UC inverse or the Moore-Penrose inverse, and the two inverses 
cannot be applied separately to distinct subsets of the state variables because 
all of the variables mix under the transformation. This can be seen from a
block-partition:
\begin{eqnarray}
   \begin{array}{rcl}
      \Amn & = &
         \left[ \begin{array}{cc} {\bf W} & {\bf X}\\
                             {\bf Y}& {\bf Z}\end{array} \right]
                       \begin{array}{l} {\bf \}}~m \\{\bf \}}~n \end{array}\\
                        & & 
                                \begin{array}{c}
                                   \;\underbrace{\;}_m\;\underbrace{\;}_n
                                \end{array}
   \end{array}
\end{eqnarray}
and noting that consistency in this case requires a generalized inverse that satisfies:
\begin{equation}
   \label{trblk}
   \tinv{(\TR_1\cdot\Amn\cdot\TR_2)} ~=~
         \tinv{\left[ 
                  \begin{array}{cc} \Dm_1{\bf W}\Dm_2 & \Dm_1{\bf X}\Rm_2\\
                                                 \Rm_1{\bf Y\Dm_2}& \Rm_1{\bf Z}\Rm_2
                  \end{array} \right]}
        ~=~ \inv{\TR_2}\cdot\Atinv\cdot\inv{\TR_1}~. 
\end{equation}
In the case of nonsingular $\Amn$ the partitioned inverse is unique:
\begin{equation}
   \inv{\Amn} ~=~
   \left[
     \begin{array}{cc}
     ({\bf W}-{\bf X}{\bf Z}^{-1}{\bf Y})^{-1} & -{\bf W}^{-1}{\bf X}({\bf Z}-{\bf Y}{\bf W}^{-1}{\bf X})^{-1} \\
      -{\bf Z}^{-1}{\bf Y}({\bf W}-{\bf X}{\bf Z}^{-1}{\bf Y})^{-1} & ({\bf Z}-{\bf Y}{\bf W}^{-1}{\bf X})^{-1}
     \end{array}
    \right]
\end{equation}
and is unconditionally consistent. Respecting the block constraints implicit from Eq.(\ref{trblk}), the desired
generalized inverse for singular $\Amn$ under the present assumptions can be verified as:
\begin{equation}
   \Atinv ~=~
   \left[
     \begin{array}{cc}
     \ginv{({\bf W}-{\bf X}\pinv{{\bf Z}}{\bf Y})} & -\ginv{{\bf W}}{\bf X}\pinv{({\bf Z}-{\bf Y}\ginv{{\bf W}}{\bf X})} \\
      -\pinv{{\bf Z}}{\bf Y}\ginv{({\bf W}-{\bf X}\pinv{{\bf Z}}{\bf Y})} & \pinv{({\bf Z}-{\bf Y}\ginv{{\bf W}}{\bf X})}
     \end{array}
    \right] .
\end{equation}
The general case involving different assumptions for more than two subsets of state variables 
(possibly different for the left and right spaces of the transformation) can be solved analogously
with appropriate partitioning.

The generalized inverse of Theorem~\ref{genginv} is unique when instantiated
using the construction defined by Lemma~\ref{gendef} by virtue of
the uniqueness of both the Moore-Penrose inverse and the scaling of
Theorem~\ref{rztheorem} (alternative scalings are discussed in
Appendix~\ref{altsca}). What is most important for present 
purposes is that the approach for obtaining  
unit-consistent generalized inverses can be efficiently
applied to a wide variety of other matrix decompositions and
operators (including other generalized matrix inverses) to impose
unit consistency. Some specific
examples are briefly considered in the following sections.

\section{UC Example: Estimating Linearized Approximations
         to Nonlinear Transformations}
\label{application}

A fundamental problem in applications ranging from
system identification to machine learning is to ascertain
an estimate of an unknown nonlinear transformation 
based on a given sample set of input and output 
vectors~\cite{haber,haykin,nelles}. More
specifically, given an $m\times n$ matrix $\Xm$ with
columns representing a domain set of $n$ $m$-dimensional
input vectors, and a $q\times n$ matrix $\Ym$ with columns 
representing the range set of $p$-dimensional
output vectors, i.e., each column $\Ym_i$ gives $f(\Xm_i)$
for the unknown transformation $f(\cdot)$, determine a
$q\times m$ matrix $\Fm$ that linearly 
approximates the unknown nonlinear mapping as:
\begin{equation}
   \Fm\Xm ~ \approx ~ \Ym.
\end{equation}
If $\Xm$ and $\Ym$ are both square matrices of full
rank then the unique solution $\Fm=\Ym\Xmi$ is 
entirely structurally determined, i.e.,
there are insufficient samples to reveal any
information beyond a simple linear mapping.
As the number of samples increases, i.e., as $n$
becomes increasingly larger than $m$, information
about the unknown nonlinear mapping from $\Xm$ to $\Ym$
becomes available but a unique $\Xmi$ no longer
exists. 

Use of the Moore-Penrose inverse 
provides a solution $\FP=\Ym\pinv{\Xm}$, and its
estimation/prediction properties can be analyzed
in terms of the consistency conditions that it
preserves. Specifically, $\FP$ is invariant with
respect to the application of a unitary transformation 
$\Um$ to the rows of $\Xm$ and $\Ym$ as
$\Xm\Um$ and $\Ym\Um$:
\begin{eqnarray}
   \FP & = & (\Ym\Um)\pinv{(\Xm\Um)} \\
   ~     & = & \Ym\Um\Umi\Xpinv \\
   ~     & = & \Ym\Xpinv
\end{eqnarray} 
and therefore preserves right-unitary
consistency:
\begin{equation}
   \FP(\Xm\Um) ~ = ~ (\FP\Xm)\Um.
\end{equation} 
This consistency with respect to arbitrary
unitary linear combinations of the input vectors
(i.e., the columns of $\Xm$) implies that
$\FP$ is insensitive to functional 
dependencies which exist purely between
corresponding columns of $\Xm$ and $\Ym$.
By contrast, the solution $\FR=\Ym\Xrinv$ obtained
using a right unit-consistent inverse is invariant
with respect to an arbitrary nonzero scaling
of individual columns of $\Xm$ in the form
of a diagonal matrix $\Dm$:
\begin{eqnarray}
   \FR & = & (\Ym\Dm)\rinv{(\Xm\Dm)} \\
   ~     & = & \Ym\Dm\Dmi\Xrinv \\
   ~     & = & \Ym\Xrinv.
\end{eqnarray} 
and therefore preserves right-{\em diagonal}
unit consistency:
\begin{equation}
   \FR(\Xm\Dm) ~ = ~ (\FR\Xm)\Dm.
\end{equation} 
instead of right-{\em unitary} consistency.

Diagonal consistency does not enforce consistency
with respect to a linear mixing of the input vectors
and is therefore strongly sensitive to the functional 
relationship defined by the $\Xm_i$ and $\Ym_i$ pairs. 
This sensitivity can be demonstrated empirically by
applying a nonlinear transformation to the columns
of a randomly generated $m\times n$ matrix $\Xm$
to produce a matrix $\Ym$ and determining the 
percentage of predictions $\FR\Xm_i$ that satisfy
\begin{equation}
    \norm{\FR\Xm_i-\Ym_i} ~ < ~ \norm{\FP\Xm_i-\Ym_i}
\end{equation} 
for a given {\em vector} norm $\norm{\cdot}$. Figure~1
shows averaged results for a series of such tests. Each
test consists of a randomly-generated nonlinear 
transformation applied to the columns of an 
$m\times n$ matrix $\Xm$ of univariate-normal
i.i.d-sampled elements. The transformation for 
each test is a degree-$m$ polynomial function 
of the elements of $\Xm_i$ with univariate-normal 
i.i.d-sampled term coefficients. The percentage
of vectors $\Xm_i$ satisfying 
$\norm{\FR\Xm_i-\Ym_i}_1 <  \norm{\FP\Xm_i-\Ym_i}_1$,
is then determined. The results show that $\FR$ 
provides more accurates estimates than $\FP$ as
the dimensionality $m$ increases, i.e., as more
information about the nonlinear function becomes
available.

%\fbox{\includegraphics[options]{image}}
\begin{figure}[h] % float placement: (h)ere, page (t)op, page (b)ottom, other (p)age
  \centering
\begin{tabular}{c|c|c}
%  \hline
  {\footnotesize {\em 3D Nonlinear Mapping}} & 
  {\footnotesize {\em 5D Nonlinear Mapping}} & 
  {\footnotesize {\em 7D Nonlinear Mapping}} \\
  \includegraphics[width=2.0in,keepaspectratio]{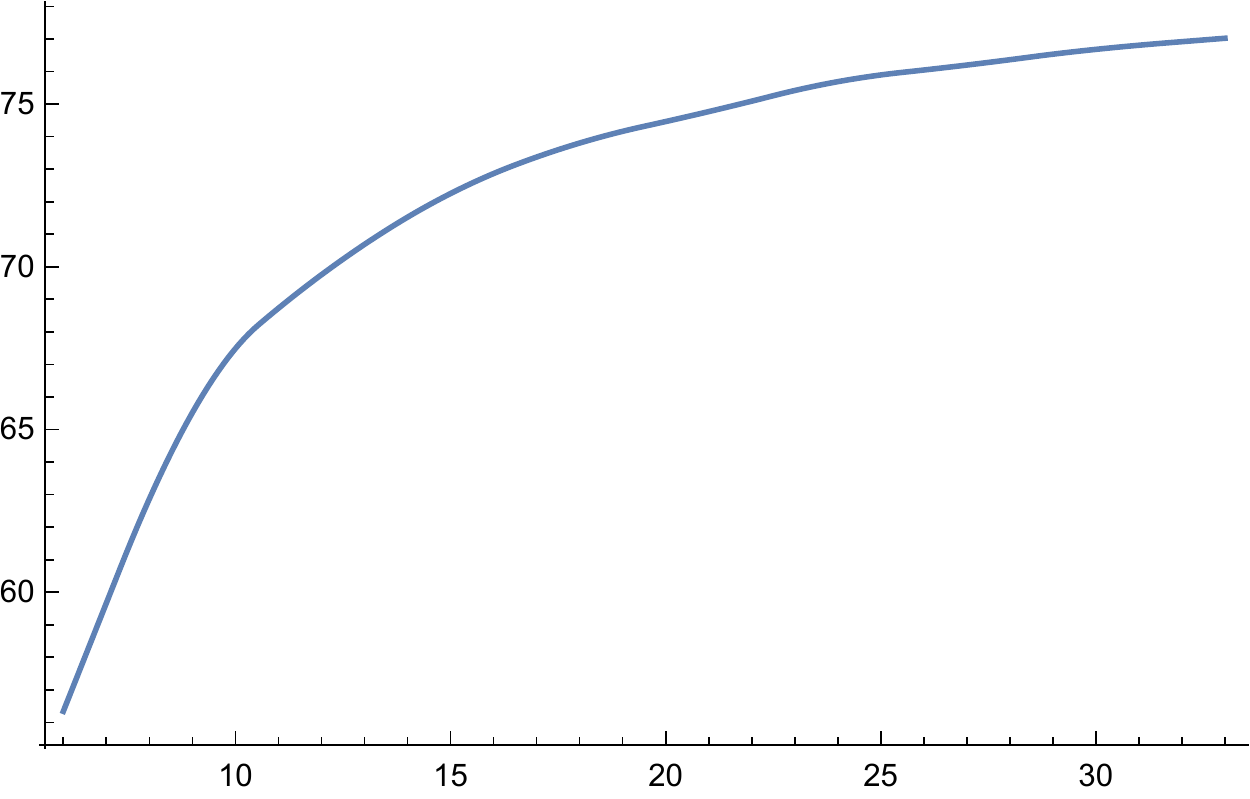} ~  & ~
  \includegraphics[width=2.0in,keepaspectratio]{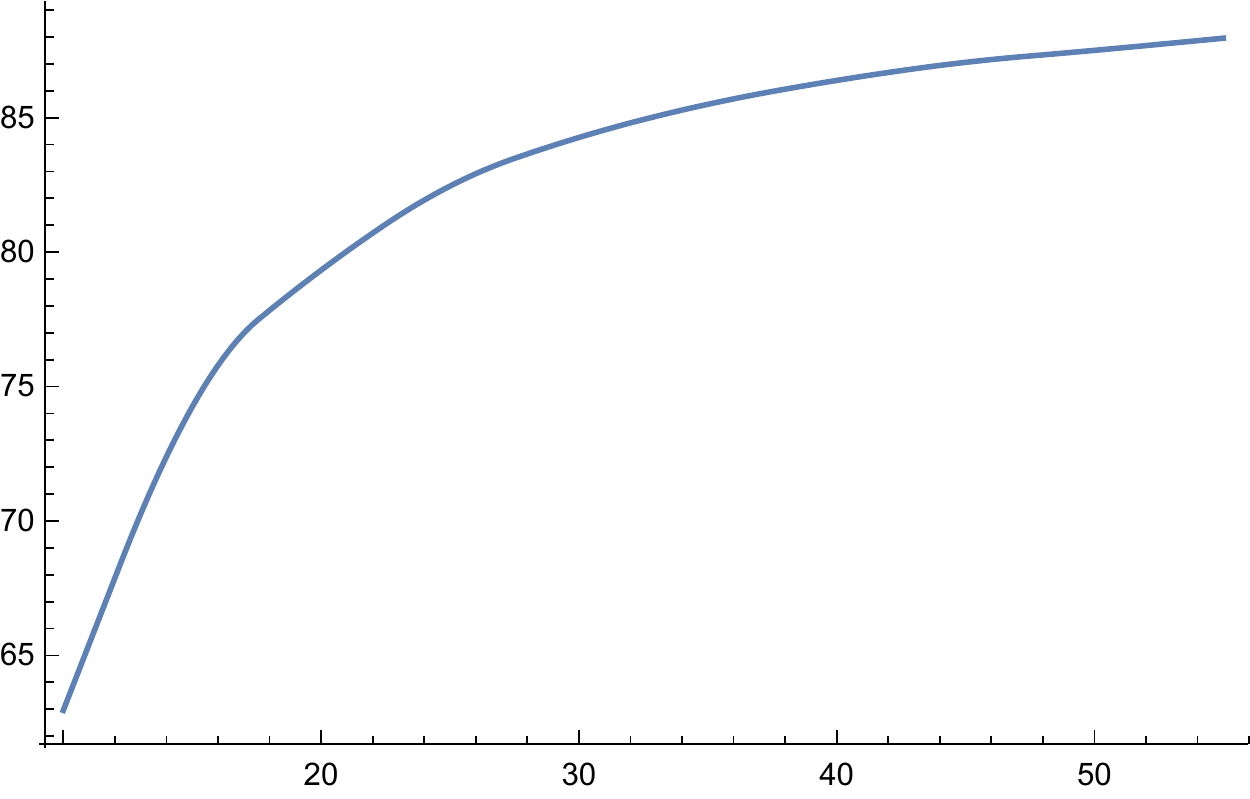} ~  & ~
  \includegraphics[width=2.0in,keepaspectratio]{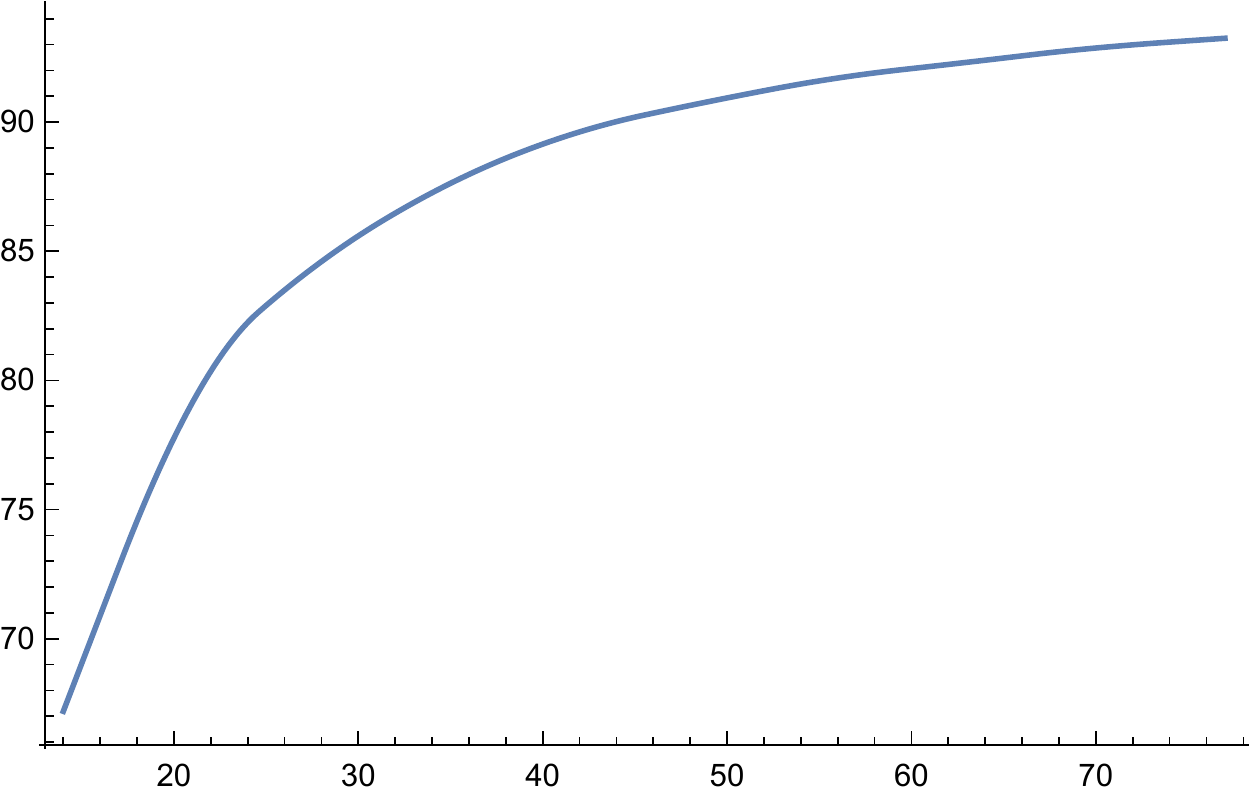} ~ ~ \\ 
%  \hline
\end{tabular}
  \caption{Relative predictive accuracy ($\ell_1$ norm) of $\FR$ versus $\FP$ for 
     random nonlinear polynomial transformations of $m$-dimensional vectors
     for $m=3$, $5$, and $7$. 
    The percentage (vertical axis) of $\FR$ estimates with error smaller than 
    their corresponding $\FP$ estimates increases with the number of samples
    $n$  (horizontal axis).}
  \label{fig:Figure1}
\end{figure}

Figure~2 shows that for sufficiently large $n$ the predictive superiority
of $\FR$ over $\FP$ as shown in Figure~1 does not depend stongly on the 
particular choice of vector norm. 
 
%\fbox{\includegraphics[options]{image}}
\begin{figure}[h] % float placement: (h)ere, page (t)op, page (b)ottom, other (p)age
  \centering
\begin{tabular}{c|c|c}
%  \hline
  {\footnotesize {\em $\ell_1$ Measure of Accuracy}} & 
  {\footnotesize {\em $\ell_2$ Measure of Accuracy}} & 
  {\footnotesize {\em $\ell_\infty$ Measure of Accuracy}} \\
  \includegraphics[width=2.0in,keepaspectratio]{R-7-1.pdf} ~  & ~
  \includegraphics[width=2.0in,keepaspectratio]{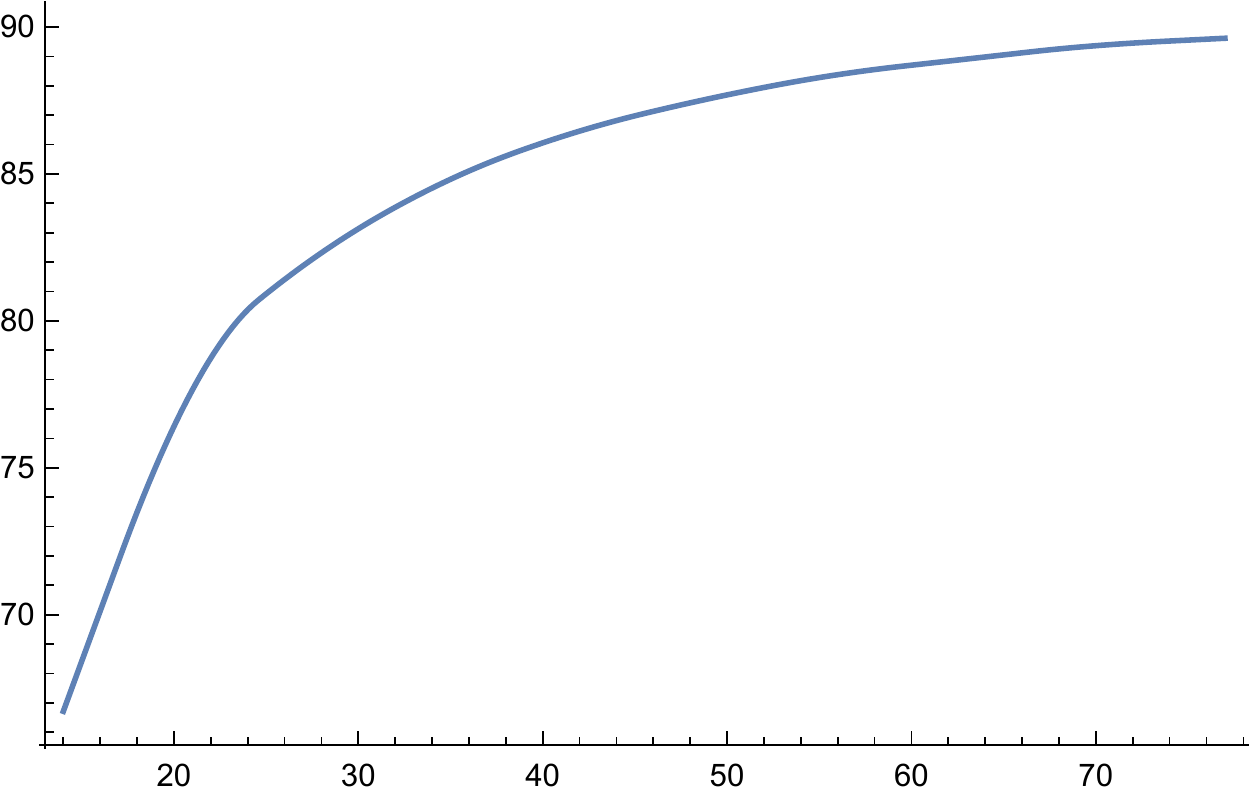} ~  & ~
  \includegraphics[width=2.0in,keepaspectratio]{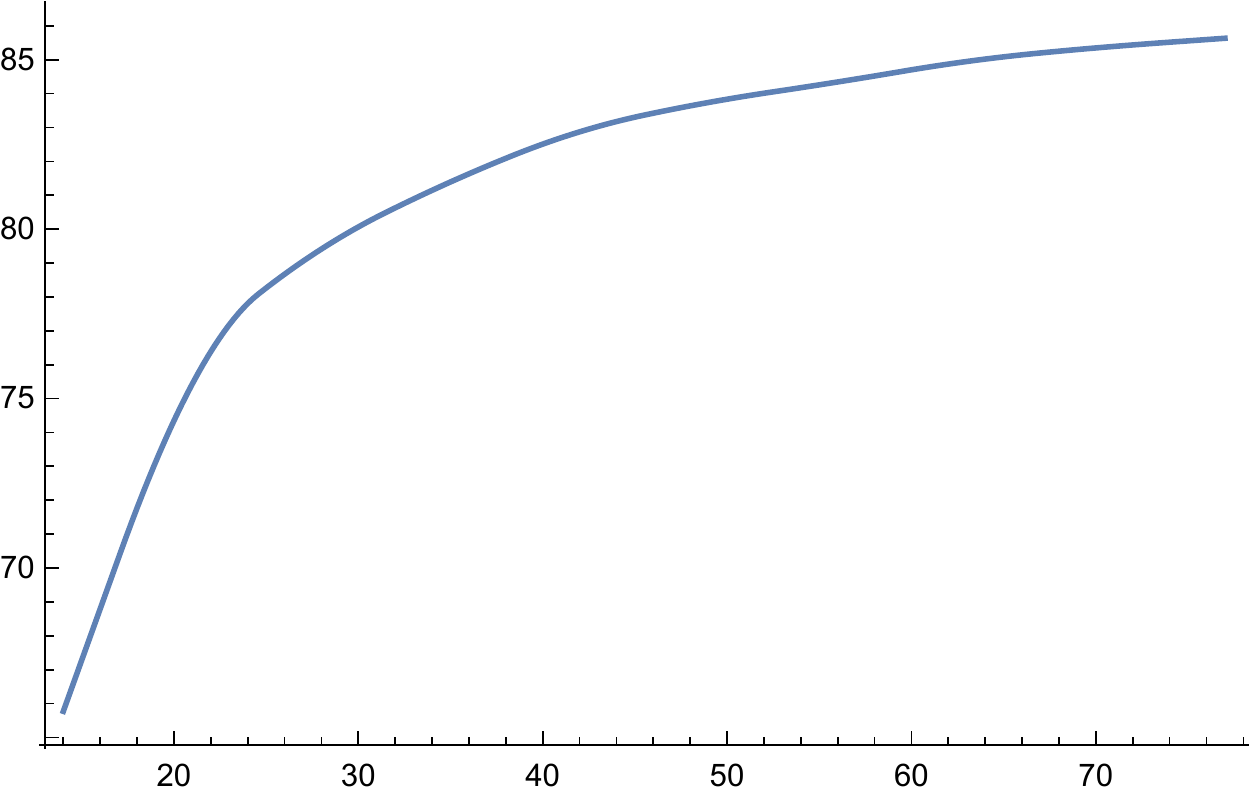} ~ ~ \\ 
%  \hline
\end{tabular}
  \caption{Relative $\ell_1$, $\ell_2$, and $\ell_\infty$ predictive 
     accuracy of $\FR$ versus $\FP$ for $m=7$.}
  \label{fig:Figure2}
\end{figure}

Although the Moore-Penrose inverse is commonly associated 
with least-squares error minimization, Figure~2 shows that
$\FR$ approaches uniform superiority even according to the 
$\ell_2$ vector norm. This is true because superiority is assessed 
here based on per-vector error rather than an average per-element
error over all predictions. It must be emphasized, however, that 
the motivation for this example
is not to provide a particular solution to a particular optimization
problem but rather to show how consistency analysis can illuminate 
salient properties of a given problem.

\section{Unit-Consistent/Invariant Matrix Decompositions}
\label{ucsvd}

The motivation to investigate unit consistency in the
context of generalized matrix inverses extends also to
other areas of matrix analysis. This clearly includes 
transformations $T[\Amn]$ which can be redefined in 
UC form as 
\begin{equation} 
   \inv{\DGL[\Amn]}\cdot 
   T\left[ \DGL[\Amn]\cdot\Amn\cdot\DGR[\Amn]\right]
   \cdot\inv{\DGR[\Amn]}
\end{equation}
and functions $f[\Amn]$ which can be redefined in unit 
scale-invariant form as
$f[\DGL[\Amn]\cdot\Amn\cdot\DGR[\Amn]]$, but it also
extends to matrix decompositions. 

The Singular Value Decomposition (SVD) is among the most
powerful and versatile tools in linear algebra and 
data analytics~\cite{ytt11,michelena93,ltp10,abb00}.
The Moore-Penrose generalized inverse of $\Amn$ can be obtained 
from the SVD of $\Amn$
\begin{equation}
      \Amn ~=~ \Um\Sm\Vmi
\end{equation}
as
\begin{equation}
      \Apinv ~=~ \Vm\tinv{\Sm}\Umi
\end{equation}
where $\Um$ and $\Vm$ are unitary, $\Sm$ is the diagonal matrix
of singular values of $\Amn$, and $\tinv{\Sm}$ is the matrix obtained from
inverting the nonzero elements of $\Sm$. This motivates the
following definition.

\begin{definition}
The Unit-Invariant Singular-Value Decomposition
(UI-SVD) is defined as
\begin{equation}
   \Amn ~=~ \Dm\cdot\Um\Sm\Vmi\cdot\Em
\end{equation}
with $\Dm=\inv{\DGL[\Amn]}$, $\Em=\inv{\DGR[\Amn]}$, and
$\Um\Sm\Vmi$ is the SVD of $\Xm=\DGL[\Amn]\cdot\Amn\cdot\DGR[\Amn]$.
The diagonal elements of $\Sm$ are referred to as the 
{\em unit-invariant} (UI) {\em singular values} of $\Amn$. 
\end{definition}

Given the UI-SVD of a matrix $\Amn$
\begin{equation}
      \Amn ~=~ \Dm\cdot\Um\Sm\Vmi\cdot\Em
\end{equation}
the UC generalized inverse of $\Amn$
can be expressed as
\begin{equation}
     \Aginv ~=~ \Emi\cdot\Vm\tinv{\Sm}\Umi\cdot\Dmi.
\end{equation}
Unlike the singular
values of $\Amn$, which are invariant with respect
to arbitrary left and right unitary transformations of
$\Amn$, the UI singular values
are invariant with respect to arbitrary left and right nonsingular
diagonal transformations\footnote{A {\em left} 
UI-SVD can be similarly defined as $\Amn  =  \Dm\cdot\Um\Sm\Vmi$
with $\Dm=\inv{\DL[\Amn]}$ and 
$\Um\Sm\Vmi$ being the SVD of $\Xm=\DL[\Amn]\cdot\Amn$.
The resulting {\em left unit-invariant singular values} are
invariant with respect to nonsingular left diagonal transformations
and right unitary transformations of $\Amn$ (the latter property
is what motivated the right unitary-invariance requirement
of Definition~\ref{defdl} and the consequent use of a
unitary-invariant norm in the construction of
Lemma~\ref{dconst}.). A {\em right}
UI-SVD can be defined analogously to interchange the
left and right invariants.}. Thus, functions of the unit-invariant
singular values are unit-invariant with respect to $\Amn$.

The largest $k$ singular values of a matrix 
(e.g., representing a photograph, video sequence,
or other object of interest) can be used to
define a unitary-invariant 
signature~\cite{czt12,gjmeyer03,hong91,luochen94,jeonglee06,jeongle09}
which supports computationally efficient similarity testing.
However, many sources of error in practical applications
are not unitary. As a concrete example, 
consider a system in which
a passport or driving license is scanned to produce
a rectilinearly-aligned and scaled image that is to 
be used as a key to search an existing image database.
The signature formed from the largest $k$
unit-invariant singular values can be used for
this purpose to provide robustness to amplitude
variations among the rows and/or columns of the
image due to the scanning process (details are
provided in the example of Section~\ref{application2}).

The UI-SVD may also offer advantages
as an alternative to the conventional 
SVD, or truncated SVD, used by existing methods
for image and signal 
processing, cryptography, digital watermarking, 
tomography, and other applications in order to provide  
state-space or coordinate-aligned robustness to 
noise\footnote{In applications in which the signature 
provided by the set of UI singular values may be too 
concise~\cite{ttw03}, e.g., because it is permutation 
invariant, a vectorization of the matrix 
$\Amn\circ(\Aginv)^{\mbox{\tiny T}}$
%$\Xm=\DGL[\Amn]\cdot\Amn\cdot\DGR[\Amn]$
can be used as a more discriminating UI
signature.}.

More generally, the approach used to define
unit scale-invariant singular values can be applied
to other matrix decompositions, though the invariance
properties may be different. In the case of 
scale-invariant eigenvalues for square $\Amn$, i.e.,
eig[$\DGL[\Amn]\cdot\Amn\cdot\DGR[\Amn]$], the
invariance is limited to diagonal transformations
$\Dm\Amn\Em$ such that $\Dm\Em$ is
nonnegative real, e.g., $\Dmp\Amn\Emp$, 
$\Dm\Amn\Dm$ (or $\Dm\Amn\bar{\Dm}$ for
complex $\Dm$), and $\Dm\Amn\Dmi$. In applications
in which changes of units can be assumed to take the
form of positive diagonal transformations, the 
scale-invariant (SI) eigenvalues can therefore be 
taken as a complementary or alternative signature 
to that of provided by the UI singular values.

\section{UI Example: Scanned-Key Image Retrieval}
\label{application2}

Most flatbed image scanners produce a digital copy of an image 
by mechanically moving a linear light source, e.g., a 
flourescent tube, horizontally across a given document/image
while continuously recording the amplitude of the 
reflected light using a CDD. The fidelity of
such a scanner can degrade over time as oil residue 
and dust accumulate along the surface of the light
source and as mechanical components begin to wear.
In particular, variations in illumination along the length 
of the light source produce horizontal amplitude 
artifacts in the resulting image while variations in the 
distance of the light source to the document due to
non-smooth motion during the scanning process  
produces vertical amplitude artifacts.

Robustness to row and column amplitude perturbations 
i.e., multiplicative noise, may be needed for
reliable image retrieval when scanned images are
to be used as keys for searching an image database. 
One mechanism for mitigating such artifacts is to 
associate a signature with every image in the form
of a vector of its largest $k$ normalized 
sorted singular values (NSVs). If $k=5$, for 
example, the NSV signature would be a vector of length 
$5$ consisting of the first five singular values.
The vector is normalized to have unit magnitude so
that the signature is invariant with respect
to a uniform scaling of pixel intensities (amplitudes),
which is necessary to accommodate global amplitude
differences among different scanners. This normalization
also provides robustness with respect to image 
decimation and super-resolution, i.e., the NSV of a 
given image will tend to be similar to the NSV of the
same image scaled to a lower or higher resolution.

Because NSV keys are nonnegative vectors, they can 
be compared using the angular distance 
metric~\cite{cossim}:
\begin{equation}
   \adist{\pv}{\qv} ~\doteq~ \frac{1}{\pi}
         \cos^{-1}\left[\frac{\pv\cdot\qv}{\norm{\pv}\norm{\qv}}\right]
\end{equation}
which simplifies in the present context to
\begin{equation}
   \adist{\pv}{\qv} ~\doteq~ \frac{1}{\pi}
         \cos^{-1}[\pv\cdot\qv]
\end{equation}
because NSV vectors are defined to have unit magnitude.
The use of a metric comparison function is needed to
permit efficient retrieval when the image database is 
implemented using a metric search 
structure~\cite{arnab,samet}.

Figure~3 shows source images: LENA, RUSSELL,
and JARRY, followed by their respective query 
images: $\sim$LENA, $\sim$RUSSELL, and $\sim$JARRY, 
which have been 
corrupted with the same set of horizontal and vertical 
variations in amplitude. Letting their respective 
5-element NSV signatures
be denoted as $\Lnsv$, $\Rnsv$, $\Jnsv$, $\LXnsv$,
$\RXnsv$, and $\JXnsv$, the effect of the scanning
artifacts can be quantified as:
\begin{eqnarray}
    \adist{\Lnsv}{\LXnsv} & = & 0.009  \\
    \adist{\Rnsv}{\RXnsv} & = & 0.005  \\
    \adist{\Jnsv}{\JXnsv} & = & 0.003
\end{eqnarray}
where a distance of zero would imply that the keys
are identical and a distance of $1$ would imply that
they are orthogonal. 

%\fbox{\includegraphics[options]{image}}
\begin{figure}[h] % float placement: (h)ere, page (t)op, page (b)ottom, other (p)age
  \centering
\begin{tabular}{cccccc}
%  \hline
  {\footnotesize {\em LENA}} & 
  {\footnotesize {\em RUSSELL}} & 
  {\footnotesize {\em JARRY}} &
  {\footnotesize {\em $\sim$LENA}} & 
  {\footnotesize {\em $\sim$RUSSELL}} & 
  {\footnotesize {\em $\sim$JARRY}} \\
  \includegraphics[width=0.85in,keepaspectratio]{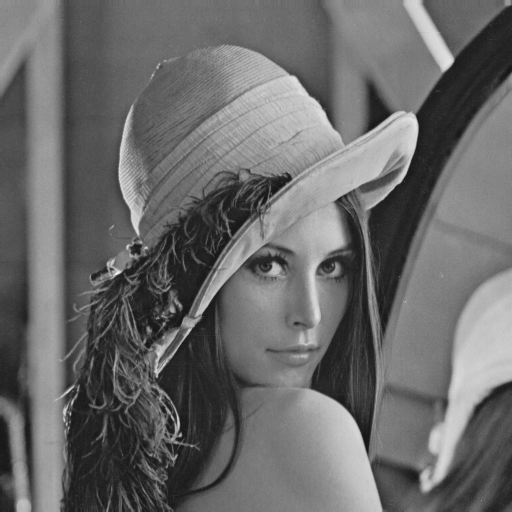} ~  & ~
  \includegraphics[width=0.85in,keepaspectratio]{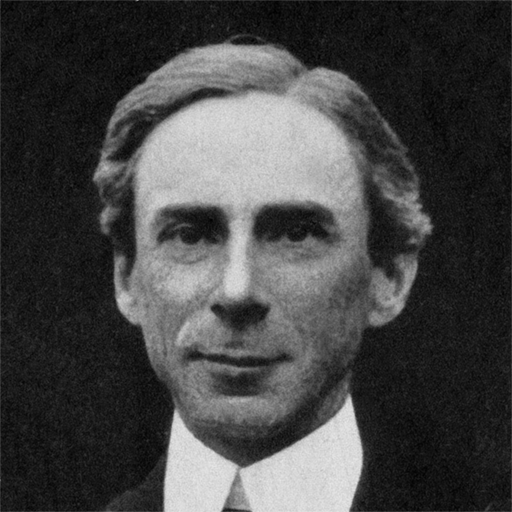} ~  & ~
  \includegraphics[width=0.85in,keepaspectratio]{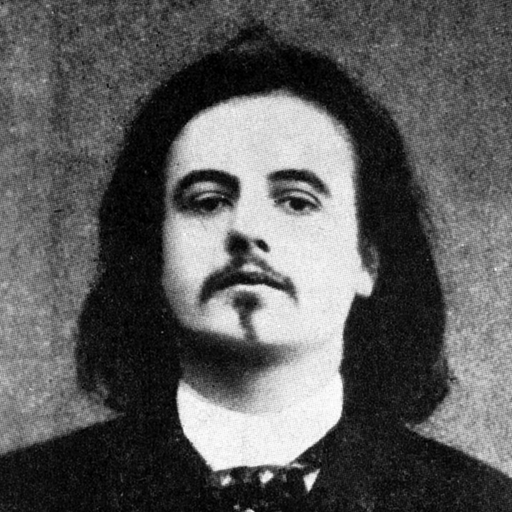} ~  & ~ 
  \includegraphics[width=0.85in,keepaspectratio]{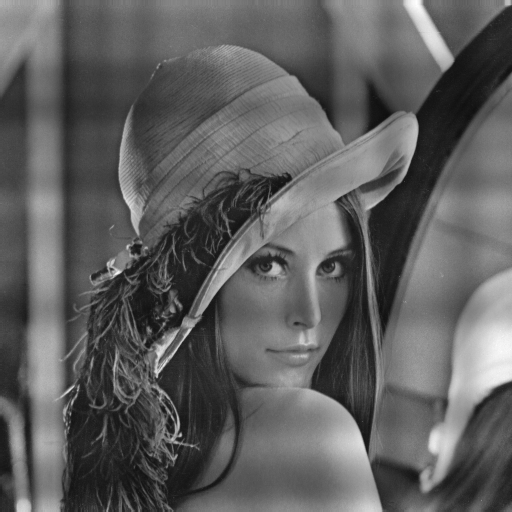} ~  & ~
  \includegraphics[width=0.85in,keepaspectratio]{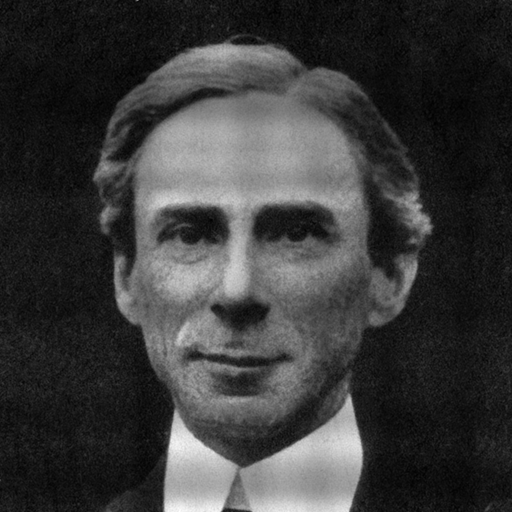} ~  & ~
  \includegraphics[width=0.85in,keepaspectratio]{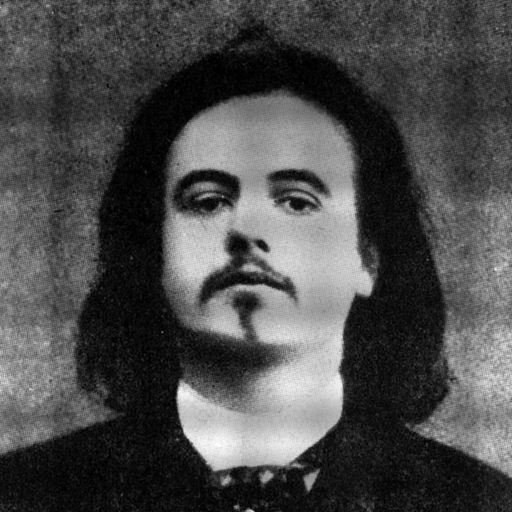} ~ ~ \\ 
%  \hline
\end{tabular}
  \caption{The first three images (left to right) represent high-quality 
     original source images, e.g., as might be stored in an image 
     database, followed by versions of each that have
     been corrupted with horizontal and vertical scaling artifacts.
     The three corrupted images were subjected to the same
     horizontal and vertical scaling process, e.g., as might be
     expected if obtained from the same poor-quality flatbed
     scanner.}
  \label{fig:Figure3}
\end{figure}

\noindent Unit-invariant singular values are 
invariant with respect to the scaling of rows and columns 
of a matrix and therefore should be invariant with respect
to the scanning artifacts depicted in Figure~3.  
This motivates an alternative
to NSVs in which singular values are replaced 
with UI singular values. As should be expected, the resulting 
UNSV keys are significantly less sensitive to amplitude
artifacts and are not identically zero only because 
of noise due to 8-bit discretization of pixel values:
\begin{eqnarray}
    \adist{\Lunsv}{\LXunsv} & = & 1.48\times 10^{-5}~  \\
    \adist{\Runsv}{\RXunsv} & = & 3.41\times 10^{-5}~  \\
    \adist{\Junsv}{\JXunsv} & = & 2.73\times 10^{-4}.
\end{eqnarray}
However, the effectiveness of a signature must be assessed
in terms of both specificity and discrimination. In this
example the latter can be seen by examining the extent to
which each source image can be distinguished from 
signatures derived from a different image. The following
shows the distance of the NSV signature for each source image
to the NSV signatures obtained from scans of different 
(i.e., non-matching) images:
%\vspace{-22pt}
%\begin{singlespacing}
\begin{eqnarray}
    \adist{\Lnsv}{\RXnsv} & = & 0.021 \\
    \adist{\Lnsv}{\JXnsv} & = & 0.048 \\
    \adist{\Rnsv}{\LXnsv} & = & 0.020 \\
    \adist{\Rnsv}{\JXnsv} & = & 0.038 \\
    \adist{\Jnsv}{\LXnsv} & = & 0.048 \\
    \adist{\Jnsv}{\RXnsv} & = & 0.036 
\end{eqnarray}
as compared to those obtained from UNSV signatures:
\begin{eqnarray}
    \adist{\Lunsv}{\RXunsv} & = & 0.080~ \\
    \adist{\Lunsv}{\JXunsv} & = & 0.171~ \\
    \adist{\Runsv}{\LXunsv} & = & 0.080~ \\
    \adist{\Runsv}{\JXunsv} & = & 0.108~ \\
    \adist{\Junsv}{\LXunsv} & = & 0.171~ \\
    \adist{\Junsv}{\RXunsv} & = & 0.109. 
\end{eqnarray}
%\end{singlespacing}
The UNSV signatures provide much larger differences
in distance between correct and incorrect pairings than
NSV signatures. For example, the UNSV signatures are
much more effective at distinguishing the image of 
British mathematician Bertrand Russell from that of 
French absurdist Alfred Jarry. These differences are 
critical because the choice of tolerance $\epsilon$ 
for searching the database must also accommodate 
image perturbations other than amplitude variations. 
Thus, if images are not sufficiently distinguished 
the retrieval process will return a large fraction 
of the images in the database as being potential 
matches.

The artifacts in this example impact the 
effectiveness of NSV signatures in terms of
both specificity and discrimination. In the 
case of specificity, the distance between
the source image LENA and its scanned
counterpart, $\sim$LENA, is relatively large:
\begin{eqnarray}
    \adist{\Lnsv}{\LXnsv} & = & 0.009 
\end{eqnarray}
and this is due entirely to the presence 
amplitude deviations. Because these same
artifacts are present in every scan, the
NSV signatures for LENA and $\sim$RUSSELL
interpret the artifacts as being features
that the two images have in common, hence
the distance between them -- i.e., the
ability to discriminate them -- is reduced:
\begin{eqnarray}
    \adist{\Lnsv}{\RXnsv} & = & 0.021\,. 
\end{eqnarray}
More specifically, the NSV distance of LENA to 
$\sim$LENA
and the NSV distance of LENA to $\sim$RUSSELL
differ by only $0.012$, whereas the
corresponding UNSV difference is almost a 
factor of $7$ larger.

SVD-based keys can of course only be used
for coarse-grain similarity and discrimination 
testing, e.g., as query keys for searching
a database, but they provide a good example
of how consistency and invariance considerations
can be applied to mitigate the effects of a
known source error such as multiplicative 
noise\footnote{It should be noted that 
invariance with respect to additive noise
can be homomorphically~\cite{smithbook} 
obtained from the
elemental logarithms of the diagonal scaling
values used to compute the UI singular
values.}. It must be emphasized that the 
method guarantees invariance with respect 
to row/column multiplicative noise artifacts, 
so the example of the three images was needed
only to illustrate this fact.

\section{Discussion}
\label{discsec}

A stated purpose of this paper is to promote unit consistency
as a system design principle so that the functional integrity of 
complex systems can be sanity-checked in a manner that is 
relatively general and application independent. 
For this principle to be 
applied in practice it is necessary to establish that 
unit-consistent and/or unit-invariant methods exist to
address a broad range of real-world engineering problems.
It is of course impossible to enumerate and consider
every kind of problem, but it has been argued that 
the techniques applied to develop unit-consistent 
generalized matrix inverses, unit-invariant SVD, etc., 
are applicable to a wide variety of problems for 
which unit consistency is commonly sacrificed through 
the reflexive use of least-squares
and other non-UC optimization criteria.

~\\
\noindent At a high level, consistency testing of a system can 
be summarized most generally as follows:
%\begin{singlespacing}
\begin{enumerate}
\item {\em A valid though otherwise arbitrary input $\xb$ is 
     provided to the system to produce an output} $\yb$:
     \begin{equation}
        \xb ~\rightarrow~ \fbox{\mbox{SYSTEM}}~\rightarrow~\yb.
     \end{equation}
\item {\em If the system is assumed to be consistent with respect 
     to some transformation $T(\xb)$ then}:
     \begin{equation}
        T(\xb) ~\rightarrow~ \fbox{\mbox{SYSTEM}}~\rightarrow~
        \yb'\neq T(\yb)~\implies~\mbox{\em Fault}.
     \end{equation}
\end{enumerate}
%\end{singlespacing}
Of course the practical
application of unit-consistency testing to large-scale systems will
almost certainly have to accommodate modules
-- or even subsystems -- that functionally should
maintain unit consistency but do not. For example,
a module for calculating an orbital trajectory may require
inputs to be provided in specific units because 
internally-used gravitational and other constants 
are defined in those units. To test a system that includes
such a module it is necessary to implement a wrapper to
convert from the testing coordinates to the module-required 
coordinates and back again. This expenditure of 
effort can be justified in that it 
allows the non-UC character of the module to be explicitly 
recognized before testing rather than being identified 
later following a failed unit-consistency test of the 
system. In other words, the UC testing process 
facilitates the explicit identification of all sources of
unit inconsistency, whether before testing or as a
result of testing.

It must be emphasized that the motivation for 
enforcing unit consistency is not limited
to simply avoiding coordinate-mismatch faults,
e.g., like that which felled the Mars Climate 
Orbiter~\cite{mco}. Rather, it is to provide 
a means for identifying a broad range of 
design and implementation flaws for which a
violation of unit consistency is just a 
side-effect. 

More generally, it is hoped that a greater focus 
on the consistency properties
of engineering solutions will yield
additional performance and reliability benefits
across a diverse spectrum of applications.

\appendices

\section{Uniqueness of the UC Inverse}
\label{facta}

By virtue of the uniqueness of the Moore-Penrose inverse, the UC inverse from 
Theorem~\ref{genginv} is uniquely determined given a scaling $\Amn=\DGL\Xm\DGR$ 
produced according to Theorem~\ref{rzcor}. However, the positive diagonal matrices $\DGL[\Amn]$ 
and $\DGL[\Amn]$ are not necessarily unique, so there may exist distinct positive
diagonal matrices $\Dma$ and $\Dmb$ and $\Ema$ and $\Emb$ such that
\begin{equation}
     \Amn ~=~ \Dma\Xm\Ema ~=~ \Dmb\Xm\Emb. 
\end{equation}
What remains is to establish the uniqueness of $\Aginv$ in this case, i.e., that
\begin{equation}
    \Dma\Xm\Ema = \Dmb\Xm\Emb ~\implies~
   \Emia\Xpinv\Dmia = \Emib\Xpinv\Dmib.
\end{equation}

We begin by noting that if an arbitrary $m\times n$ matrix $\Amn$ has rank $r$ then it can be factored~\cite{big} as the product
of an $m\times r$ matrix $\Fm$ and an $r\times n$ matrix
$\Gm$ as
\begin{equation}
\Amn ~ = ~ \Fm\Gm
\end{equation}
The Moore-Penrose inverse can then be expressed in terms of
this rank factorization as
\begin{equation}
\Apinv ~ = ~ \Gmt\cdot\inv{(\Fmt\cdot\Amn\cdot\Gmt)}\cdot\Fmt \label{rfac}
\end{equation}
where $\Gmt$ and $\Fmt$ are the conjugate transposes
of $\Gm$ and $\Fm$. 

Because $\Dma\Xm\Ema=\Dmb\Xm\Emb$ implies
\begin{equation}
     \Xm ~=~ \Dmib\Dma\Xm\Ema\Emib 
\end{equation}
then from the rank factorization $\Xm=\Fm\Gm$ 
we can obtain an alternative factorization
\begin{equation}
     \Xm ~=~ \Fm'\Gm' ~=~ (\Dmib\Dma\Fm)(\Gm\Ema\Emib)
\end{equation}
from the fact that the ranks of $\Fm$ and
$\Gm$ are unaffected by nonsingular diagonal
scalings. Applying the rank factorization identity
for the Moore-Penrose inverse then yields
\begin{eqnarray}
     \Xpinv & = & \pinv{(\Fm'\Gm')} \\
~ & = & (\Gm\Ema\Emib)^*\cdot
           \inv{ \left( (\Dmib\Dma\Fm)^*\cdot\Xm\cdot (\Gm\Ema\Emib)^* \right) }
                        \cdot (\Dmib\Dma\Fm)^* \\
~ & = & \Ema\Emib\Gmt\cdot
                \inv{ \left( (\Fmt\Dmib\Dma)\Xm(\Ema\Emib\Gmt) \right) }
                        \cdot\Fmt\Dmib\Dma \\
~ & = & (\Ema\Emib)\cdot\Gmt\cdot
                \inv{ \left( \Fmt\cdot(\underline{\Dmib\Dma\Xm\Ema\Emib})\cdot\Gmt \right) }
                        \cdot\Fmt\cdot(\Dmib\Dma) \\
~ & = & (\Ema\Emib)\cdot\left(\Gmt\cdot
                \inv{ \left( \Fmt\Xm\Gmt \right) }\cdot
                        \Fmt\right)\cdot(\Dmib\Dma) \\
~ & = & \Ema\Emib\left(\underline{\Gmt\cdot\inv{(\Fmt\Xm\Gmt)}\cdot\Fmt}\right)\Dmib\Dma \\
~ & = & \Ema\Emib\Xpinv\Dmib\Dma 
\end{eqnarray}
which implies\footnote{Note that the diagonal matrices commute and are real so, e.g., 
$\Dm^*=\Dm$.} 
\begin{eqnarray}
\Emia\Xpinv\Dmia & = & \Emia\cdot\left(\Ema\Emib\Xpinv\Dmib\Dma\right)\cdot\Dmia \\
~ & = & (\Emia\Ema)\cdot\Emib\Xpinv\Dmib\cdot(\Dma\Dmia) \\
~ & = & \Emib\Xpinv\Dmib
\end{eqnarray}
and thus establishes that $\Emia\Xpinv\Dmia=\Emib\Xpinv\Dmi$ and therefore
that the UC generalized inverse $\Aginv$ is unique.

Using a similar but more involved application of rank factorization it can be 
shown that that the UC generalized matrix inverse satisfies
\begin{equation}
   \Aginv \cdot \ginv{(\Aginv)} \cdot \Aginv ~=~ \Aginv
\end{equation}
which is weaker than the uniquely-special property of the Moore-Penrose
inverse:
\begin{equation}
    \pinv{(\Apinv)}=\Amn. 
\end{equation}

\section{Alternative Constructions}
\label{altsca}

The proofs of Theorems~\ref{linvt} and and~\ref{ginvtp}
(and consequently Theorem~\ref{genginv}) do not actually 
require the general unitary consistency property of 
the Moore-Penrose inverse and instead 
only require diagonal unitary consistency, e.g.,
in Eqs.(\ref{lmpua})-(\ref{lmpub}) as
\begin{equation}
    \pinv{(\Dmu\Amn)} ~ = ~ \Apinv\Dmui
\end{equation}
and in Eqs.(\ref{gpinvua})-(\ref{gpinvub}) as   
\begin{equation}
   \pinv{(\Dmu\Amn\Emu)} ~ = ~ \Emui\Apinv\Dmui
\end{equation}
for unitary diagonal matrices $\Dmu$ and $\Emu$.
Thus, the Moore-Penrose inverse could be replaced
with an alternative which maintains the other 
required properties but satisfies this weaker 
condition in place of general unitary consistency.

Similarly, the scalings defined by
Lemmas~\ref{nonzerodef} and~\ref{gendef} 
are not necessarily the only ones
that may be used to satisfy the conditions of 
Definition~\ref{defdg}. More specifically,
Lemmas~\ref{nonzerodef} and~\ref{gendef}
define left and right nonnegative diagonal scaling functions
$\DGL[\Amn]$ and $\DGR[\Amn]$ satisfying
\begin{equation}
   \DGL[\Amn]\cdot\Amn\cdot\DGR[\Amn]  ~ = ~
   \DGL[\Dmp\Amn\Emp]\cdot\Dmp\Amn\Emp\cdot\DGR[\Dmp\Amn\Emp]
\end{equation}
for all positive diagonals $\Dmp$ and $\Emp$. Because the unitary
factors of the elements of $\Amn$ are unaffected by the nonnegative
scaling, the scalings can be constructed without loss of generality from
$\Abs[\Amn]$. If nonnegative $\Amn$ is square, irreducible, and has full 
support then such a scaling can be obtained by alternately normalizing the rows
and columns to have unit sum using the Sinkhorn iteration~\cite{sink64,sink67}.
The requirement for irreducibility stems from the fact that the process
cannot always converge to a finite left and right scaling. For example,
the matrix 
\begin{equation}
   \label{trimat}
   \begin{bmatrix}
      a & b \\ 0 &  c
   \end{bmatrix}
\end{equation} 
cannot be scaled so that the rows and columns sum to unity unless
the off-diagonal element $b$ is driven to zero, which is not possible
for any finite scaling. In other words, the Sinkhorn unit-sum
condition cannot be jointly satisfied with respect to both the set
of row vectors and the set of column vectors. What is needed,
therefore, is a measure of vector ``size'' that can be applied
within a Sinkhorn-type iteration but is guaranteed to converge to
a finite scaling\footnote{This definition and the
subsequently-defined instance, $s_{a,b}[\uv]$, may be of 
independent interest for analyzing properties
of low-rank subspace embeddings in high-dimensional vector 
spaces, e.g., infinite-dimensional spaces.}. 

\begin{definition}
For all vectors $\uv$ with elements from a normed division algebra, 
a nonnegative composable size function $s[\uv]$ is defined
as satisfying the following conditions for all $\alpha$:
\begin{eqnarray}
   s[\uv] & = & 0  ~ ~ \Leftrightarrow ~ ~ \uv ~ = ~ {\bf 0}\\
   s[\alpha\uv] & = & | \alpha | \cdot s[\uv] \\
   s[\bv] & = & 1 ~ ~ ~ ~\forall \bv \in \{0,1\}^n \\
   s[\uv] & = &  s[\uv \otimes \bv] ~ = ~ s[\bv \otimes \uv] ~ ~ ~ ~\forall \bv \in \{0,1\}^n-\znc 
%   s[\uv] & = & s[\Perm\cdot\uv] ~ ~ ~ ~ ~ ~ ~ ~\forall \Perm \in \mbox{Permutation} 
\end{eqnarray}

\end{definition}
The defined size function provides a measure of scale that
is homogeneous, permutation-invariant, and invariant with respect to tensor expansions
involving identity and zero elements. More intuitively, however, $s[\uv]$ 
can be thought of as a ``mean-like'' measure taken over the magnitudes of the nonzero 
elements of $\uv$. With the imposed condition $s[\mbox{\bf 0}]\doteq 0$
the following instantiations can also be verified to satisfy the definition: 
\begin{eqnarray}
   \stimes[\uv] & \doteq & \left(\prod_{k\in S} |\uv_k| \right)^{1/|S|} ~ ~ ~ ~ ~ ~ ~ ~ ~  j\in S ~ ~ \mbox{iff} ~ ~ \uv(j)\neq 0 \\
   s_p[\uv] & \doteq & \norm{\uv}_p ~/~ |S|^{1/p}  ~ ~  ~ ~ ~ ~ ~ ~ ~ ~ ~ ~  j\in S ~ ~ \mbox{iff} ~ ~ \uv(j)\neq 0 \\
%s_{a,b}[\uv] & \doteq &  \left(\frac{\sum_i |\uv_i|^{b}}{\sum_i |\uv_i|^a}\right)^{1/(b-a)} ~ ~ ~ ~ ~ a>0,~b>0,~a\neq b
   s_{a,b}[\uv] & \doteq &  \left(\frac{\sum_i |\uv_i|^{a+b}}{\sum_i |\uv_i|^a}\right)^{1/b} ~ ~ ~ ~ ~ ~ ~ ~ ~ ~ ~ a>0,~b>0
%   s_*[\uv] & \doteq & \uv^*\uv ~/~ \sum_i |\uv_i|
\end{eqnarray}
The first case, $\stimes(\uv)$, is more easily interpreted as the geometric mean of the 
nonzero elements of $\uv$. Its application in a Sinkhorn-type iteration converges to a 
unique scaling in which the {\em product} of the nonzero elements in each row and column has unit
magnitude. If $a$, $b$, and $c$ are positive for the matrix 
of Eq.(\ref{trimat}) then the scaled result using $s_p[\uv]$ is
\begin{equation}
   \begin{bmatrix}
      1 & 1 \\ 0 &  1
   \end{bmatrix}
\end{equation} 
where the product of the nonzero elements in each row and column is unity and the
particular left and right diagonal scalings are determined by the values of 
$a$, $b$, and $c$. It can be shown that for all elemental nonzero 
matrices that the scaling produced 
using $\stimes(\uv)$ is equivalent to that produced by the 
constructions defined by Lemmas~\ref{nonzerodef} and~\ref{gendef}
and that the iteration is fast-converging.

The row/column conditions imposed by $s_p[\uv]$ can most easily be understood
in the case of $p=1$, for which it is equivalent to the mean of the absolute values of
the nonzero elements of $\uv$. In the case of $p=2$, if a vector $\vv$ is formed 
from the $m$  nonzero elements of $\uv$ then 
\begin{equation}
        s_2[\uv] ~ = ~  \norm{\vv}_2 ~/ ~m^{1/2}
\end{equation}
In the example of the $2\times 2$  matrix of Eq.(\ref{trimat}) the scaled result 
produced using $s_p[\uv]$ for any $p>0$ happens to be the same as that produced
using $\stimes(\uv)$. For nontrivial matrices, however, the results for different 
$p$ are not generally (nor typically) equivalent to each other or to that 
produced by $\stimes(\uv)$. 

The third size function, $s_{a,b}[\uv]$, satisfies the required conditions
without imposing special treatment of zero elements. In other words, it is a
continuous function of the elements of $\uv$ and would therefore 
appear to be a more natural choice for instantiating $\DGL[\Amn]$ and 
$\DGR[\Amn]$ for analysis purposes, e.g., in the limit as $a$ and $b$ 
go to zero where $s_{a,b}[\uv]\equiv\stimes[\uv]$. (It should be noted
that the homogeneity properties of $s_{a,b}[\uv]$ hold generally
for any $a$ and $b$ from a normed division algebra with $0^0\doteq 1$.)

%\newpage
\section{Implementations}
\label{code}

%\begin{singlespacing}
Below are basic Octave/Matlab implementations
%\footnote{Thanks to Truc Le for various efficiency  
%suggestions.} 
of some of the methods developed in the paper.
Although not coded for maximum efficiency or numerical robustness, they should 
be sufficient for experimental corroboration of theoretically-established properties.\\
~\\
\noindent The following function computes $\Aginv$ for 
$m\times n$ real or complex matrix $\Amn$. It has complexity
dominated by the Moore-Penrose inverse calculation,
which is $O(mn\cdot\min(m,n))$.
\begin{verbatim}
function Ai = uinv(A)
    [S dl dr] = dscale(A);
    Ai = pinv(S) .* (dl * dr)';
end
\end{verbatim}

\noindent The following function evaluates the 
UC/UI singular values of the real or
complex matrix $\Amn$.
\begin{verbatim}
function s = usvd(A)
    s = svd(dscale(A));
end
\end{verbatim}

\noindent The following function evaluates the 
UC/UI singular-value decomposition of the $m\times n$ real or 
complex matrix $\Amn$.
\begin{verbatim}
function [D U S V E] = usv_decomp(A)
    [S dl dr] = dscale(A);
    D = diag(1./dl);  E = diag(1./dr);
    [U S V] = svd(S);
end
\end{verbatim}

\noindent The following function computes the unique
(up to unitary factors) scaled matrix
$\Sm=\DGL[\Amn]\cdot\Amn\cdot\DGR[\Amn]$
with diagonal left and right scaling matrices
$\DGL[\Amn]=\mbox{diag[dl]}$ and
$\DGR[\Amn]=\mbox{diag[dr]}$. It has 
$O(mn)$ complexity for $m\times n$ real or 
complex matrix $\Amn$.
\begin{verbatim}
function [S dl dr] = dscale(A)
    tol = 1e-15;    
    [m, n] = size(A);
    L = zeros(m, n);    M = ones(m, n);
    S = sign(A);   A = abs(A);
    idx = find(A > 0.0);
    L(idx) = log(A(idx));
    idx = setdiff(1 : numel(A), idx);
    L(idx) = 0; A(idx) = 0; M(idx) = 0;   
    r = sum(M, 2);   c = sum(M, 1);   
    u = zeros(m, 1); v = zeros(1, n);
    dx = 2*tol;  
    while (dx > tol)
        idx = c > 0;
        p = sum(L(:, idx), 1) ./ c(idx);
        L(:, idx) = L(:, idx) - repmat(p, m, 1) .* M(:, idx);
        v(idx) = v(idx) - p;  dx = mean(abs(p));
        idx = r > 0;
        p = sum(L(idx, :), 2) ./ r(idx);
        L(idx, :) = L(idx, :) - repmat(p, 1, n) .* M(idx, :);
        u(idx) = u(idx) - p;  dx += mean(abs(p));
    end    
    dl = exp(u);   dr = exp(v);
    S .*= exp(L);
end
\end{verbatim}
%\end{singlespacing}

%\newpage

\end{document}